\documentclass[a4paper,12pt]{article}

\usepackage{amsfonts,amssymb,amsmath,dsfont,relsize,accents}
\usepackage{amsthm}
\usepackage[english]{babel}
\usepackage[small,labelfont=it,margin=1em]{caption}
\usepackage{color}
\usepackage[utf8x]{inputenc} 
\usepackage{latexsym}
\usepackage{mathrsfs}
\usepackage{upgreek}
\usepackage{graphicx}
\usepackage{hyperref}
\usepackage[normalem]{ulem}

\DeclareMathAlphabet{\mathdutchcal}{U}{dutchcal}{m}{n}
\SetMathAlphabet{\mathdutchcal}{bold}{U}{dutchcal}{b}{n}
\DeclareMathAlphabet{\mathdutchbcal}{U}{dutchcal}{b}{n}
\newcommand*{\s}{{\mathdutchcal{s}}}
\newcommand*{\p}{{\mathdutchcal{p}}}
\newcommand*{\h}{{\mathdutchcal{h}}}
\newcommand*{\dt}{{\mathdutchcal{t}}}
 \newcommand{\sst}{s_\star}
 \newcommand{\tpb}{\mathscr{A}}
 \newcommand{\barm}{\bar{M}}

\newtheorem{theorem}{Theorem}[section]
\newtheorem*{theorem*}{Theorem}
\newtheorem{lemma}{Lemma}[section]
\newtheorem{proposition}[lemma]{Proposition}
\newtheorem{corollary}[lemma]{Corollary}
\theoremstyle{definition}
\newtheorem*{definition*}{Definition}
\newtheorem*{definition}{Definition}

\newcommand*{\E}{\mathbb{E}}

\newcommand*{\N}{\mathbb{N}}
\renewcommand{\P}{\mathbb{P}}

\newcommand*{\R}{\mathbb{R}}

\renewcommand{\leq}{\leqslant}
\renewcommand{\le}{\leqslant}
\renewcommand{\geq}{\geqslant}
\renewcommand{\ge}{\geqslant}
\renewcommand{\subset}{\subseteq}

\hypersetup{colorlinks=true,urlcolor=blue,linkcolor=black,citecolor=black}

\newcommand*{\doi}[1]{\href{http://dx.doi.org/\detokenize{#1}}{doi}}

\begin{document}

\title{Multi-range percolation on oriented trees: critical curve and limit behavior}

\author{Bernardo N. B. de Lima\textsuperscript{1}, R\'eka Szab\'o\textsuperscript{2},  Daniel Valesin\textsuperscript{3}}
\footnotetext[1]{Universidade Federal de Minas Gerais. \url{bnblima@mat.ufmg.br}}
\footnotetext[2]{CEREMADE, CNRS, Universit\'e Paris-Dauphine, PSL University. \url{szabo@ceremade.dauphine.fr}}
\footnotetext[3]{Bernoulli Institute, University of Groningen. \url{d.rodrigues.valesin@rug.nl}}
\date{February 26, 2021}
\maketitle

\begin{abstract}
We consider an inhomogeneous oriented percolation model introduced by de Lima, Rolla and Valesin~\cite{LRV17}. In this model, the underlying graph is an oriented rooted tree in which each vertex points to each of its~$d$ children with ``short'' edges, and in addition, each vertex points to each of its~$d^k$ descendant at a fixed distance~$k$ with ``long'' edges. A bond percolation process is then considered on this graph, with the prescription that independently, short edges are open with probability~$p$ and long edges are open with probability~$q$. We study the behavior of the critical curve~$q_c(p)$: we find the first two terms in the expansion of~$q_c(p)$ as~$k \to \infty$, and prove that the critical curve lies strictly above the critical curve of a related branching process, in the relevant parameter region. We also prove limit theorems for the percolation cluster in the supercritical, subcritical and critical regimes. 
\end{abstract}

{\bf\large{}}\bigskip


\section{Introduction}
\subsection{Background and motivation}
This paper is a continuation of the work presented in~\cite{LRV17}.
In that paper, the authors considered an oriented graph whose vertex set is that of the~$d$-regular, rooted tree, containing ``short edges'' (with which each vertex points to its~$d$ children) and, for some~$k \in \mathbb{N}$ fixed, ``long edges'' of range~$k$ (with which each vertex points to its~$d^k$ descendants~$k$ generations below). 
Percolation is defined on this graph by letting short edges be open with probability~$p$ and long edges with probability~$q$.
For all fixed~$q$ one can define the critical percolation threshold as the supremum of the values of~$p$ for which there is almost surely no infinite cluster at parameters~$p, q$. 
The authors of~\cite{LRV17} study the properties of this critical curve and prove monotonicity with respect to the length of the long edges.

The work was originally motivated by the following problem.
Consider the graph having~$\mathbb{Z}^d$ as vertex set and all edges of the form~$\{x, x\pm e_i\}$ and~$\{x, x\pm k\cdot e_i\}$ for some~$k\geq 2$ and~$i\in\{1,\dots,d\}$.
It was shown in \cite{LSS11} that the critical probability for Bernoulli bond percolation on this graph converges to that of~$\mathbb{Z}^{2d}$ as~$k\rightarrow\infty$.
This result was later generalized in \cite{MT17}.
The convergence is conjectured to be monotone, that is, the percolation threshold for the above graph should be decreasing in the length~$k$ of long edges.

Percolation is mostly studied on the lattice $\mathbb{Z}^d$ in homogeneous environment (that is, each edge is open with the same probability independently of each other).
Let us briefly mention some related works that consider an inhomogeneous setting. 
In~\cite{KSV12} an oriented site percolation model is considered on~$\mathbb{Z}^2_+$ in a random environment. Each line $l_i:=\{(x,y)\in\mathbb{Z}_+^2: x+y=i\}$ is declared to be bad with probability~$\delta$, then the sites on bad lines are open with probability~$p_B$ and every other site is open with probability~$p_G$. It is shown that for all~$p_G>p_c(\mathbb Z^2_+)$ and~$p_B>0$ we can choose~$\delta>0$ small enough, that there is an infinite cluster with positive probability. Another interesting paper with the same spirit is \cite{DHKS18}, about Brochette percolation. A bond percolation model is considered on~$\mathbb{Z}^2$ where vertical lines of the form~$\bar l_i:=\{(x,y)\in\mathbb{Z}_+^2: x=i\}$ are selected at random with probability~$\delta$, then the edges on the selected lines are open with probability~$p$ and every other edge is open with probability~$q$. The authors show that for all~$p>p_c(\mathbb{Z}^2)$ and any~$\delta > 0$ we can choose~$q<p_c(\mathbb{Z}^2)$ such that there is an infinite cluster with positive probability.

In~\cite{IRM15} the authors study a non-oriented bond percolation model on~$\mathbb Z^d$ with an~$s$-dimensional defect plane~$\mathbb{Z}^s$. Edges of~$\mathbb Z^d\setminus\mathbb Z^s$ are open with probability~$p$ and edges of~$\mathbb Z^s$ are open with probability~$\sigma$. They present the phase diagram of the model and identify three regimes in which the model exhibits quantitatively different behaviour. They also show that the critical curve is a strictly decreasing function for~$p$ with a jump discontinuity at~$p_c(\mathbb Z^d)$.

In~\cite{GN90} the authors consider non-oriented percolation on the direct product of a regular tree and~$\mathbb{Z}$, where tree edges and line edges are open with different probabilities. They identify three distinct phases in which the number of clusters is~$0, \infty$ and 1 respectively.

In~\cite{SV18} for an arbitrary connected graph~$G=(V, E)$ a non-oriented and an oriented percolation model are defined on the vertex set~$V \times \mathbb{Z}$.
The authors examine how changing the percolation parameter on a fixed (infinite) set of edges affects the critical behaviour.
They show that in both cases the critical parameter changes as a continuous function of these parameters. This result was later generalized in~\cite{LS19}.

\subsection{Description of the model and results}

Given~$d, k \in \{2, 3, \dots\}$, define an oriented graph~$\mathbb{T} = \mathbb{T}_{d,k} = (V,E)$ in the following way.
Denote~
\begin{align*}
[d]=\{1, \dots, d\}, \qquad [d]_*=\bigcup_{0\leq n < \infty}[d]^n.
\end{align*}
The set~$[d]^0$ consists of a single point~$o$, which we will refer to as the root of the graph. Set~$V=[d]_*$, that is, elements of~$V$ are sequences~$v=(v_1, \dots, v_m)$ with~$v_i\in[d]$ (and the root~$o$). Define the concatenation of~$u=(u_1, \dots, u_m)$ and~$v=(v_1, \dots, v_n)$ as
\begin{align*}
u \cdot v &= (u_1, \dots, u_m, v_1, \dots, v_n);\\
v \cdot o &= o\cdot v=v.
\end{align*}
Further let~$E=E_{\s}\cup E_{\ell}$ be the set of oriented edges with
\begin{align*}
E_{\s}=\{\langle r, r \cdot i \rangle: r\in V,\; i\in[d]\},\qquad
E_{\ell}=\{\langle r, r \cdot i \rangle: r\in V,\; i\in[d]^k\}.
\end{align*}
We will refer to these sets as the set of ``short'' and ``long'' edges, respectively.
Define the out-degree of a vertex as the number of oriented edges directed out of the vertex.
Note that in~$\mathbb{T}_{d,k}$ every vertex has out-degree~$d+d^k$.

Consider the following percolation model on~$\mathbb{T}$: every edge in~$E_{\s}$ is open with probability~$p$, and every edge in~$E_{\ell}$ is open with probability~$q$, independently of each other.
Denote the law of the model by~$\mathbb{P}_{p, q}$.
We will omit the subscript~$p, q$ when it is clear from the context; we will also generally omit~$d$ and~$k$ from the notation. Define the cluster of the root~$\mathscr{C} = \mathscr{C}_{p,q}$ as the set of vertices that can be reached by an oriented open path from~$o$, and~$|\mathscr{C}|$ as its cardinality. 
Whether or not the event~$\{|\mathscr{C}_{p,q}|=\infty\}$ occurs with positive probability depends on~$p, q, d$ and~$k$. We define
$$q_c = q_c(p) := \inf\{q: \; \P(|\mathscr{C}_{p,q}| = \infty) > 0\}.$$
Note that the model with~$q=0$ reduces to oriented percolation on a~$d$-ary tree, and the model with~$p = 0$ reduces to oriented percolation on disconnected~$d^k$-ary trees. This shows that~$q_c(0) = d^{-k}$ and~$q_c(p) = 0$ for~$p > d^{-1}$. 

In~\cite{LRV17} it was shown that~$q_c(p)$ is continuous and strictly decreasing in the region where it is positive (namely, for~$p \in[0,d^{-1}])$, and there is almost surely no infinite cluster for~$p \in [0,d^{-1}]$ and~$q =q_c(p)$.

Additionally, it is easy to see that our percolation model is stochastically dominated by a  branching process with offspring distribution that is the sum of two independent binomial random variables, namely~$\text{Bin}(d, p)$ and~$\text{Bin}(d^k, q)$. This branching process is critical for parameters satisfying~$dp + d^kq= 1$. This shows that~$q_c(p) \ge (1-dp)/d^k$ for~$p\in (0,d^{-1})$.
\begin{figure}[htb]
\begin{center}
{\includegraphics[width = 0.5\textwidth]{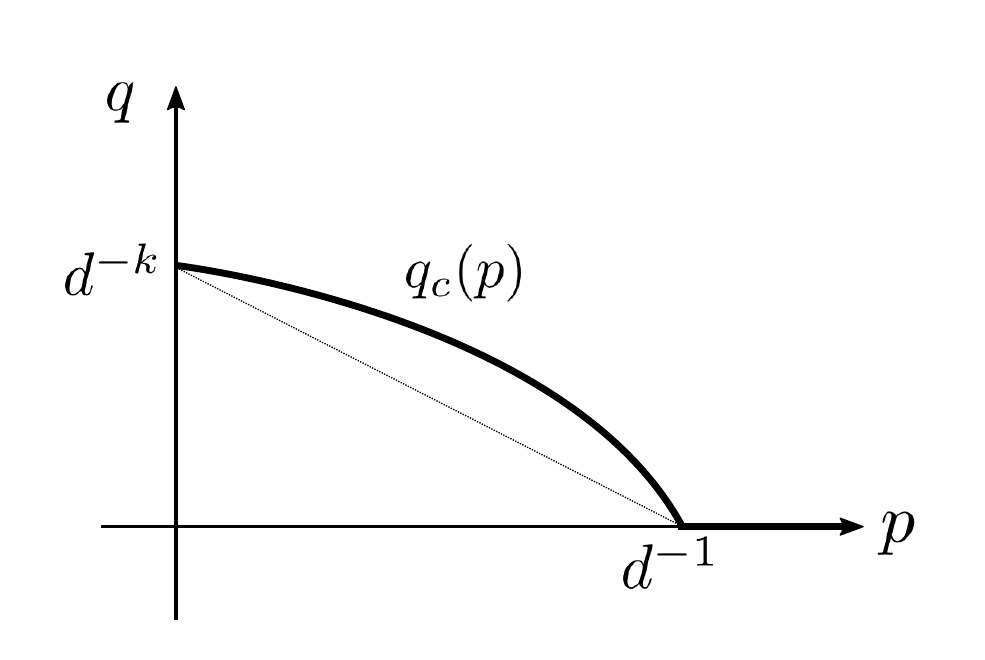}}
\end{center}
\caption{$(0,d^{-1})\ni p\mapsto q_c(p)$ lies strictly above the dotted line, which has equation~$dp + d^kq= 1$. }
\label{fig:pcq}
\end{figure}
 Our first result is the statement that this inequality is strict, as depicted in Figure~\ref{fig:pcq}.
\begin{theorem}\label{THM:CRITCURVE}
	For every~$p\in(0,d^{-1})$ we have $q_c(p)> (1-dp)/d^k$.
\end{theorem}

Next, we study the asymptotic behavior of~$q_c(p)$ as the length~$k$ of long edges is taken to infinity, when~$p < d^{-1}$ (that is, there is no percolation using only short edges). Since we have that~$q_c(p) \le q_c(0) = d^{-k}$ for any~$p \ge 0$, we readily obtain that~$q_c(p) \xrightarrow{k \to \infty} 0$. We find the two first asymptotic terms in this convergence.
\begin{theorem} \label{thm:qck} Assume~$pd < 1$. Then, as~$k \to \infty$,
	\begin{equation*}
	q_c(k) = (1-pd)\cdot \frac{1}{d^k} + \frac{(1-pd)^2p^2d}{1-p^2d}\cdot \frac{1}{d^{2k}} + o\left(\frac{1}{d^{2k}}\right).
	\end{equation*}
\end{theorem}

Our proofs rely on comparisons between exploration processes of our percolation cluster, on the one hand, and multi-type branching processes, on the other hand.{ With such comparisons at hand, many results from the theory of multi-type branching processes can be applied to the percolation cluster; in the following theorem we collect a few that we find particularly noteworthy.} Let~$X_n$ denote the number of vertices of height~$n$  in~$\mathscr{C}_{p,q}$.
\begin{theorem}\label{thm:limits}
\begin{itemize}
\item[$\mathrm{(i)}$] If~$q > q_c(p)$, then there exists a constant ~$\rho=\rho(p, q)$ and a nonnegative random variable~$Y$ such that~$\mathbb P(Y>0)>0$ and 
 \[
  \lim_{n\rightarrow\infty}\frac {X_n}{\rho^n}=Y  \text{ a.s.}
 \]
\item[$(\mathrm{ii})$] If~$q < q_c(p)$, then for each~$i \ge 0$ the limit
$$P(i) :=\lim_{n \to \infty} \P(X_n = i\mid X_n \neq 0)  $$
exists. Moreover,~$\sum_{i=0}^\infty P(i) = 1$.
\item[$\mathrm{(iii)}$] Assume that~$p \in [0,d^{-1}]$ and~$q = q_c(p)$. Then,~$\mathscr{C}_{p,q}$ rooted at~$o$ and conditioned on having more than~$n$ vertices converges locally as~$n \to \infty$  to a random rooted graph (in the Benjamini-Schramm sense~\cite{BS01}). 
\end{itemize}
\end{theorem}
We will recall the meaning of Benjamini-Schramm local convergence of rooted graphs in Section~\ref{ss:limit_theorems}.

\subsection{Discussion of proofs and organization of the paper}
The proofs of each of our theorems rely on comparisons with branching processes (multi-type branching processes for Theorems~\ref{thm:qck} and~\ref{thm:limits}, and single-type branching processes for Theorem~\ref{THM:CRITCURVE}). In {Section~\ref{sect:branchingproc}} we give an overview of multi-type branching processes, with the aim of fixing notation, listing common regularity and positivity assumptions on the offspring distribution, and reviewing key related objects (the mean offspring matrix and its Perron-Frobenius eigenvalue).

In {Section~\ref{sect:mainthm}} we prove Theorem~\ref{thm:qck}. The main idea is to explore the percolation cluster by repeating a two-stage procedure which  takes as input a set~$B$ of vertices (all of which have a common ancestor at distance smaller than~$k$ away), first explores its ``short cluster'' (that is, the set of vertices that can be reached from~$B$ through open short edges), and secondly reveals all open long edges that start at this short cluster. The endpoints of these open long edges are then grouped into sets~$B_1,B_2,\ldots$, to each of which the procedure is applied again. A similar exploration method was employed in~\cite{LRV17}. The contribution here is to take the analysis of this exploration further by studying a multi-type branching process, where the set of types is the set of all possible ``shapes'' of the input set. A careful analysis of the mean offspring matrix of this branching process yields upper and lower bounds to the threshold~$q_c(p)$.

In {Section~\ref{sect:limitthm}} we prove Theorem~\ref{thm:limits}. This relies on a much simpler exploration method, which reveals the percolation cluster by incrementing its height one unit at a time. The three parts of Theorem~\ref{thm:limits} are almost direct applications of corresponding branching process limit theorems, which we recall in that section.

Finally, {Section~\ref{sect:critcurve}} contains the proof of Theorem~\ref{THM:CRITCURVE}. The method we employ is a direct comparison between the percolation cluster and the branching process that we have briefly described before stating Theorem~\ref{THM:CRITCURVE}. We show that if we take the branching process with parameters~$(p,q-\delta)$ (for~$\delta$ small), it still stochastically dominates the percolation cluster with parameters~$(p,q)$. This is achieved by means of a useful coupling technique (reproduced as Lemma~\ref{lemma:coupling} below) which has been fruitful in other contexts~\cite{LRV17,LS19, SV18}.

\section{Basic facts about multi-type branching processes}\label{sect:branchingproc}

In this section, we give some basic definitions and notations concerning multi-type branching processes. We refer the reader to~\cite{AN72} for a detailed introduction to this topic.

We let~$\mathscr{T}$ be a finite set, representing the  space of types for our branching process. Elements of~$\mathscr{T}$ will be denoted by lower-case letters such as~$a$ and~$b$. The state space of a multi-type branching process with space of types~$\mathscr{T}$ is the set~$(\N_0)^{\mathscr{T}}$; elements of this set will be denoted by Greek letters such as~$\eta$ or~$\xi$. The interpretation is that~$\eta \in (\N_0)^{\mathscr{T}}$ corresponds to a population with~$\eta(a)$ individuals of type~$a$, for each~$a \in \mathscr{T}$. We let~$\mathbf{e}_a \in (\N_0)^\mathscr{T}$ denote a population with a single individual of type~$a$.
We also denote by~$\mathbf{0}$  the identically zero element of~$(\N_0)^\mathscr{T}$.

 Next, let~$\p:\mathscr{T}\times (\mathbb{N}_0)^\mathscr{T} \to \mathbb{R}$ be a function such that \[\p(a,\eta) \ge 0 \;\forall a \in \mathscr{T},\;\eta \in (\N_0)^\mathscr{T}\quad \text{and}\quad \sum_{\eta \in (\N_0)^\mathscr{T}} \p(a,\eta)=1 \;\forall a \in \mathscr{T}. \]  We interpret~$\p(a,\eta)$ as the probability that an individual of type~$a$ at generation~$n$ is replaced by the population~$\eta$ at generation~$n+1$.  In order to avoid certain pathological situations, it is common to assume that~$\p$ is such that at least one of the types has positive probability of generating a population with more than one individual, that is,
\begin{equation}\label{eq:not_rw}\tag{$\star$}\p(a,\eta) > 0 \text{ for some } a \in \mathscr{T} \text{ and } \eta \in (\N_0)^\mathscr{T} \text{ with } \sum_b \eta(b) > 1.\end{equation}

A multi-type branching process with space of types~$\mathscr{T}$ and offspring distribution~$\p$ is a Markov chain~$\mathbf{Z}=(\mathbf{Z}_n)_{n \ge 0}$ on~$(\N_0)^\mathscr{T}$ with transition function described as follows.
 Given that~$\mathbf{Z}_n = \eta$, the distribution of~$\mathbf{Z}_{n+1}$ is equal to the law of
$$\sum_{a \in \mathscr{T}} \sum_{\ell=1}^{\eta(a)} \chi_{a,\ell},$$
where~$(\chi_{a,\ell}:a\in \mathscr{T},\;\ell\in \N)$ are independent random variables  with~$\chi_{a,\ell} \sim \p(a,\cdot)$ for each~$a,\ell$ (and the second summation above should be understood  as zero if~$\eta(a) = 0$).
We refer to~$\mathbf{Z}_n$ as the population at generation~$n$.

We say that the multi-type branching process~$\mathbf{Z}$ goes extinct if the event that~$\mathbf{Z}_n = \mathbf{0}$ for some~$n$ occurs; otherwise we say that it survives. We denote its survival probability by
\[\zeta(a) = \P(\mathbf{Z} \text{ survives} \mid \mathbf{Z}_0 = \mathbf{e}_a),\quad a \in \mathscr{T}.\]
With some abuse of notation, we can define a function from~$(\N_0)^\mathscr{T}$ to~$\R$, also denoted~$\zeta$, by letting
\begin{equation}\label{eq:def_of_zeta}\zeta(\eta):=\P(\mathbf{Z} \text{ survives}\mid \mathbf{Z}_0 = \eta) = 1-\prod_{a \in \mathscr{T}}(1-\zeta(a))^{\eta(a)}, \quad \eta \in (\N_0)^\mathscr{T}.\end{equation}

We next define the mean offspring matrix
\begin{equation} \label{eq:matrix_M}M(a,b):=\E[\mathbf{Z}_1(b) \mid \mathbf{Z}_0 = \mathbf{e}_a],\qquad a,b \in \mathscr{T}.\end{equation}
A common assumption about this matrix is that
\begin{equation}
\label{eq:pos} \tag{P} \text{there exists } n \in \N \text{ such that }M^n(a,b) > 0 \text{ for all }a,b \in \mathscr{T}.
\end{equation}
If~$M$ satisfies~\eqref{eq:pos}, then by the Perron-Frobenius theorem it has a maximal eigenvalue~$\rho$ which is positive and simple.

For the rest of this section, we assume that both~\eqref{eq:not_rw} and~\eqref{eq:pos} hold; moreover, all the multi-type branching processes that we consider in the rest of the paper will satisfy both assumptions. When they are in place, the following holds. If~$\rho\leq1$ we have~$\zeta(a)=0$ for all~$a\in \mathscr{T}$, whereas  if~$\rho > 1$ we have~$\zeta(a) >0$ for all~$a \in \mathscr{T}$ (see for example Chapter~V of \cite{AN72}).
We call the process supercritical, critical or subcritical if~$\rho>1$,~$\rho=1$ or~$\rho<1$ respectively.

Finally, let us mention that, in place of a stochastic process~$(\mathbf{Z}_n)_{n \ge 0}$ on~$(\N_0)^\mathscr{T}$, a natural way to represent a multi-type branching process is by means of a (random, marked, unoriented) tree, which we will call the family tree of the process, denoted~$\tau$. The construction is as follows: each individual of type~$a$ in~$\mathbf{Z}$ is represented by a vertex with mark~$a$ in~$\tau$, and an (unoriented) edge is placed for each pair~$\{\text{parent, child}\}$. In case the initial population of~$\mathbf{Z}$ has a single individual, then~$\tau$ consists of a single connected tree (which is infinite if and only if~$\mathbf{Z}$ survives). More generally, an initial population with~$m$ individuals would give rise to~$m$ disconnected trees, but we will not need to consider this case. The random tree~$\tau$ is sometimes referred to as a multi-type Galton-Watson tree.

\section{Asymptotics on length of long edges}\label{sect:mainthm}
The goal of this section is proving Theorem~\ref{thm:qck}. Before going back to our percolation model on~$\mathbb{T}$, we prove some results about transformations on multi-type branching processes, and the effects of these transformations on the survival probability. This is done in Section~\ref{ss:further_mtbp}; there, the main result we are after is Lemma~\ref{lem:main_comparison_bp} below. We will then appeal to this result when analyzing the mean offspring matrix associated to a branching process that arises when exploring the percolation cluster, in Section~\ref{ss:first_bpropc}.
\subsection{Preliminary results on branching processes}\label{ss:further_mtbp}
We will now prove some auxiliary results concerning two comparison processes obtained from a multi-type branching process.

Let~$\mathbf{Z}=(\mathbf{Z}_n)_{n \ge 0}$ be a multi-type branching process with set of types~$\mathscr{T}$ and offspring distribution~$\p$. Let~$I \subsetneq \mathscr{T}$ be a strict subset of the set of types.  Define a new multi-type branching process~$\mathbf{X}=(\mathbf{X}_n)_{n \ge 0}$ with same set of types~$\mathscr{T}$, and offspring distribution~$\tilde{\p}$ given as follows. For each~$a \in \mathscr{T}$, the distribution~$\tilde{\p}(a,\cdot)$ on~$(\N_0)^{\mathscr{T}}$ is the distribution of
$$\sum_{b \notin I} \chi(b)\cdot \mathbf{e}_b + \sum_{b \in I} \sum_{\ell=1}^{\chi(b)} \chi'_{b,\ell} , $$
where
$$\chi \sim \p(a,\cdot) \quad \text{and}\quad \chi'_{b,\ell} \sim \p(b,\cdot) \; \text{for all } b \in I,\; \ell \in \N,$$
all these variables being taken independently. In words, this is described as follows. Fix an individual of type~$a$ in some generation of~$\mathbf{X}$. Then, to sample the population obtained as offspring of this individual, we first sample~$\chi \sim \p(a,\cdot)$, according to the law of the offspring of type~$a$ in the original process~$\mathbf{Z}$. The portion of the population~$\chi$ whose type does not belong to~$I$ is left unaltered. Next, each individual in~$\chi$ of type~$b \in I$ is replaced by a further offspring sampled (independently) according to the law~$\p(b,\cdot)$. All these populations are then combined together. 

The following can be easily proved with the aid of the family tree of~$\mathbf{Z}$; we omit the details.
\begin{lemma} \label{lem:first_comp}
Let~$\mathbf{Z}$,~$I$ and~$\mathbf{X}$ be as above. We then have, for any~$\eta \in (\N_0)^\mathscr{T}$,
$$\P(\mathbf{Z} \text{ survives}\mid \mathbf{Z}_0 = \eta) = \P(\mathbf{X} \text{ survives}\mid \mathbf{X}_0 = \eta).$$
\end{lemma}

We now turn to the description of our second comparison process.  Again let~$\mathbf{Z}$ be a multi-type branching process with set of types~$\mathscr{T}$ and offspring distribution~$\p$. Let~$a_\star \in \mathscr{T}$ be a distinguished type, and assume that there exists a function~$\lambda:\mathscr{T} \to \mathbb{N}$ with~$\lambda({a_\star})=1$ and such that
\begin{equation}\label{eq:ineq_on_l}
\P(\mathbf{Z} \text{ survives} \mid \mathbf{Z}_0 = \mathbf{e}_a) \le \P(\mathbf{Z} \text{ survives} \mid \mathbf{Z}_0 = \lambda(a)\cdot \mathbf{e}_{a_\star}) \quad \text{ for all }a\in \mathscr{T},
\end{equation}
that is, the process has at least as high a chance of surviving when started from a population of~$\lambda(a)$ individuals of type~$a_\star$ than it would if started with a single individual of type~$a$. We now define a (one-type) branching process~$(Y_n)_{n \ge 0}$ with~$Y_0 = 1$ and offspring distribution equal to the law of
$$\sum_{a\in \mathscr{T}}\lambda(a)\cdot \chi(a), \quad \text{where }\chi \sim \p(a_\star,\cdot).$$
We then have:
\begin{lemma} \label{lem:second_comp}
Let~$\mathbf{Z}$,~$a_\star$,~$\lambda$ and~$(Y_n)$ be as above. If~$(Y_n)$ goes extinct with probability one,  then~$\mathbf{Z}$ started from~$\mathbf{Z}_0=\mathbf{e}_{a_\star}$ also goes extinct with probability one.
\end{lemma}
\begin{proof}
Recall the definition of~$\zeta$ from~\eqref{eq:def_of_zeta}, and define~$\bar{\zeta}(a) = 1-\zeta(a)$ for~$a \in \mathscr{T}$ and~$\bar{\zeta}(\eta) = 1 - \zeta(\eta)$ for~$\eta \in (\N_0)^\mathscr{T}$. Note that~\eqref{eq:ineq_on_l} can then be written
\begin{equation}\label{eq:zeta2}
\bar{\zeta}(a) \ge \bar{\zeta}(\lambda(a) \cdot \mathbf{e}_{a_\star}) \stackrel{\eqref{eq:def_of_zeta}}{=} \bar{\zeta}(a_\star)^{\lambda(a)}.
\end{equation}
With some abuse of notation, we lift~$\lambda$ to a function on~$(\N_0)^\mathscr{T}$ by letting
$$\lambda(\eta) := \sum_{a \in \mathscr{T}} \eta(a)\cdot \lambda(a),\quad \eta \in (\N_0)^\mathscr{T},$$
and note that, for any~$\eta$,
\begin{align*}
\bar{\zeta}(\eta) \stackrel{\eqref{eq:def_of_zeta}}{=} \prod_{a \in \mathscr{T}} (\bar{\zeta}(a))^{\eta(a)} \stackrel{\eqref{eq:zeta2}}{\ge} \prod_{a \in \mathscr{T}} (\bar{\zeta}(a_\star))^{\lambda(a)\cdot \eta(a)} = (\bar{\zeta}(a_\star))^{\lambda(\eta)}.
\end{align*}
We then have
\begin{equation}\label{eq:g_ineq}\bar{\zeta}(a_\star) = \mathbb{E}[\bar{\zeta}(\mathbf{Z}_1)\mid \mathbf{Z}_0 = \mathbf{e}_{a_\star}] \ge \mathbb{E}[(\bar{\zeta}(a_\star))^{\lambda(\mathbf{Z}_1)}\mid \mathbf{Z}_0 = \mathbf{e}_{a_\star}] = \mathbb{E}[(\bar{\zeta}(a_\star))^{Y_1}].\end{equation}

Let~$G(s) := \mathbb{E}[s^{Y_1}]$ be the probability generating function of the offspring distribution of~$(Y_n)$. It is easy to check recursively that
$$\mathbb{E}[s^{Y_n}] = G^{(n)}(s),$$
where~$G^{(n)}$ is the~$n$-fold composition of~$G$. Note that for~$s = \bar{\zeta}(a_\star)$,~\eqref{eq:g_ineq} gives~$s \ge G(s)$, which can be iterated (since~$G$ is non-decreasing) to~$s \ge G^{(n)}(s)$ for any~$n$. Hence,
$$\bar{\zeta}(a_\star) \ge G^{(n)}(\bar{\zeta}(a_\star)) = \E[(\bar{\zeta}(a_\star))^{Y_n}]  \stackrel{\eqref{eq:def_of_zeta}}{=} \E[\bar{\zeta}(Y_n\cdot \mathbf{e}_{a_\star})].$$
Now, if~$(Y_n)$ goes extinct almost surely, then~$Y_n \to 0$ almost surely as~$n \to \infty$, so  the right-hand side above converges to~$1$ as~$n \to \infty$, so~$\bar{\zeta}(a_\star) = 1$.
\end{proof}

We now give the final form of the comparison result that will be used in the following subsection. 
\begin{lemma}\label{lem:main_comparison_bp}
Let~$\mathbf{Z}=(\mathbf{Z}_n)_{n \ge 0}$ be a multi-type branching process with set of types~$\mathscr{T}$, offspring distribution~$\p$, and mean offspring matrix~$M$. Let~$a_\star \in \mathscr{T}$ and~$I \subset \mathscr{T}\backslash\{a_\star\}$, and assume that~$\mathbf{Z}_0 = \mathbf{e}_{a_\star}$.
\begin{itemize}
\item[(a)] If
\begin{equation}\label{eq:mean_of1}M(a_\star,a_\star) + \sum_{a \in I} M(a_\star,a)\cdot M(a,a_\star) > 1,\end{equation}
then~$\mathbf{Z}$ survives with positive probability;
\item[(b)] Let~$\lambda:\mathscr{T}\to \mathbb{N}$ be a function satisfying~$\lambda(a_\star)=1$ and~\eqref{eq:ineq_on_l}. Then,
\begin{equation}\label{eq:mean_of2}\sum_{a \notin I} M(a_\star,a)\cdot \lambda(a) + \sum_{a \in I} \sum_{b \in \mathscr{T}} M(a_\star,a)\cdot M(a,b)\cdot \lambda(b) < 1, \end{equation}
then~$\mathbf{Z}$ goes extinct almost surely.
\end{itemize}
\end{lemma}
\begin{proof}
For~(a), first define a process~$\mathbf{X}$ as the one of Lemma~\ref{lem:first_comp}, start it from a single individual of type~$a_\star$, and modify its offspring distribution~$\tilde{p}(a_\star,\cdot)$ so that all individuals of type different from~$a_\star$ are discarded. This gives rise to a one-type branching process which is smaller than~$\mathbf{X}$; its mean offspring is the expression on the left-hand side of~\eqref{eq:mean_of1}, so it is supercritical if this expression is larger than~1. This proves that~\eqref{eq:mean_of1} implies that~$\mathbf{X}$ survives with positive probability, so by Lemma~\ref{lem:first_comp}, the same holds for~$\mathbf{Z}$.

For~(b), take the same process~$\mathbf{X}$ as above. Note that, by Lemma~\ref{lem:first_comp}, condition~\eqref{eq:ineq_on_l} holds with~$\mathbf{X}$ in place of~$\mathbf{Z}$. We then define a one-type branching process process~$(W_n)_{n \ge 0}$ from~$\mathbf{X}$ in the same way that the process~$(Y_n)$ of  Lemma~\ref{lem:second_comp} is constructed from~$\mathbf{Z}$ (that is, in the offspring population, any individual of type~$a$ is replaced by~$\lambda(a)$ individuals of type~$a_\star$). Then, the mean offspring of~$(W_n)$ is equal to the left-hand side of~\eqref{eq:mean_of2}. If this number is smaller than~1, then~$(W_n)$ goes extinct with probability~1, so (by Lemma~\ref{lem:second_comp}) $\mathbf{X}$ goes extinct with probability~1, so (by Lemma~\ref{lem:first_comp})~$\mathbf{Z}$ goes extinct with probability~1.
\end{proof}
\subsection{First branching process representation of percolation cluster}\label{ss:first_bpropc}
We now go back to the study of our percolation model on~$\mathbb{T}_{d,k}$. For the rest of this section, we fix~$p$ and~$d$ with~$pd<1$, so that there is no percolation using only short edges.

We first introduce some notation and terminology. For a vertex~$v\in V$ define the function~$h$ to be its distance (in short edges) from the root:
\begin{align*}
h(o):=0,\qquad h((v_1,\ldots, v_m)):=m,
\end{align*}
and for a set~$A\subseteq V$ let
\begin{equation*} h(A):=\sup_{v\in A}h(v).\end{equation*}

For any~$v\in V$ let~$\mathbb{T}^v=(V^v, E^v)$ and~$\Gamma^v=(V_{\Gamma}^v, E_{\Gamma}^v)$ be the subgraphs of~$\mathbb{T}$ induced by the vertex sets
\begin{align}
V^v&=\{v\cdot s: s\in V\}, \label{eq:vt} \\
V_\Gamma^v&=\{v\cdot s: s\in V,\; h(s)<k\}, \label{eq:vgamma}
\end{align}
that is the subgraph rooted at~$v$ and the subgraph of height~$k-1$ rooted at~$v$.

We now define a property of sets of vertices that will be of interest in the following. 
\begin{definition}\label{def:admissible}
	A set~$B \subset V$  is called \emph{admissible} if there exists a (necessarily unique) vertex~$u \in B$ such that~$B \subset V^u_\Gamma$. In this case, we say that~$u$ is the \textit{base} of~$B$, and write~$u = \mathdutchcal{b}(B)$.
\end{definition}

In what follows, we let~$B \subset V$ be an admissible set. We define~$\mathscr{C}^B = \mathscr{C}^B(\omega)$ as the ``cluster of~$B$'', that is, the set of vertices that can be reached by an oriented open path from some vertex in~$B$ (with the understanding that~$B \subset \mathscr{C}^B$). Note that
\begin{equation}
\mathscr{C}^B = \bigcup_{x \in B}\mathscr{C}^{\{x\}}.
\label{eq:cluster_union}\end{equation}
The following inequality will be very useful in the sequel:
\begin{equation}\label{eq:prob_mon}
\mathbb{P}\left(|\mathscr{C}^B| < \infty\right) \ge \mathbb{P}\left( |\mathscr{C}^{\{o\}}| < \infty\right)^{|B|}.
\end{equation}
This can be seen from~\eqref{eq:cluster_union}, by revealing the clusters~$\mathscr{C}^{\{x\}}$, for~$x \in B$, one by one, and using the FKG inequality.

We now let~$\mathcal C_\s^B$ be the ``short cluster of~$B$'', that is, the set of vertices that can be reached by an oriented open path that starts from some vertex in~$B$ and uses only edges of~$E_\s$. We again adopt the convention that~$B \subset \mathcal{C}^B_\s$. Note that~$\mathcal C_\s^B\subseteq\mathscr C^B$. Now, suppose that we have revealed~$\mathcal{C}_{\s}^B$ (this involves querying short edges only) and, once this is done, we reveal all open long edges that start at some point of~$\mathcal{C}^B_{\s}$. Let~$\mathcal{C}^B_\ell$ be the set of vertices that are obtained as endpoints of these long edges. Equivalently,
\[
\mathcal C_\ell^B:=\left\{\begin{array}{l}  r\in\mathscr C^B\setminus\mathcal C_\s^B: r \text{ is the endpoint of some}\\ \text{open long edge started from }  \mathcal C_\s^B \end{array}\right\}.
\]
The following lemma is an easy consequence of the assumption that~$B$ is admissible, so we omit the proof. 
\begin{lemma}\label{lem:geom_dec}
	The following statements hold:
	\begin{itemize}
		\item[$\mathrm{(i)}$] for any~$u \in \mathcal{C}^B_\ell$, we have~$V^u \cap \mathcal{C}^B_s = \varnothing$, and
		\item[$\mathrm{(ii)}$] for any~$u,v \in \mathcal{C}^B_\ell$ with~$v \in V^u$, we have~$h(v)- h(u) < k$. 
	\end{itemize}
\end{lemma}

We will now decompose~$\mathcal{C}^B_\ell$ in a very useful manner. To do so, first declare that two vertices~$u,v \in \mathcal{C}^B_\ell$ are related if there exists~$w \in \mathcal{C}^B_\ell$ such that~$u,v \in V^w$; it is readily seen that this is an equivalence relation. Then let~$\mathcal{B}^B$ be the collection of equivalence classes induced by this relation.

\begin{lemma}\label{lem:decompositions}
	For any admissible set~$B$,
	\begin{itemize}
		\item[$\mathrm{(i)}$] every~$B' \in \mathcal{B}^B$ is admissible, and
		\item[$\mathrm{(ii)}$] the sets~$V^{\mathdutchcal{b}(B')}$, with~$B'$ ranging over~$\mathcal{B}^B$, are disjoint, and disjoint from~$\mathcal{C}^B_\s$.
\end{itemize} \end{lemma}
\begin{proof}
	It readily follows from the definition of the  equivalence relation given above that for each~$B' \in \mathcal{B}^B$ there is a unique vertex~$u_\star(B') \in B'$ such that~$B'\subset V^{u_\star(B')}$, and also that if~$B', B'' \in \mathcal{B}^B$ are distinct, then~$V^{u_\star(B')}$ and~$V^{u_\star(B'')}$ are disjoint. The statements of the lemma readily follow from these considerations and Lemma~\ref{lem:geom_dec}.
\end{proof}

For an example of a short cluster~$\mathcal{C}_\s^o$ and the decomposition of~$\mathcal{C}_\ell^o$ see Figure~\ref{fig:23pqregq}.

\begin{figure}[htb]
\begin{center}
{\includegraphics[width = \textwidth]{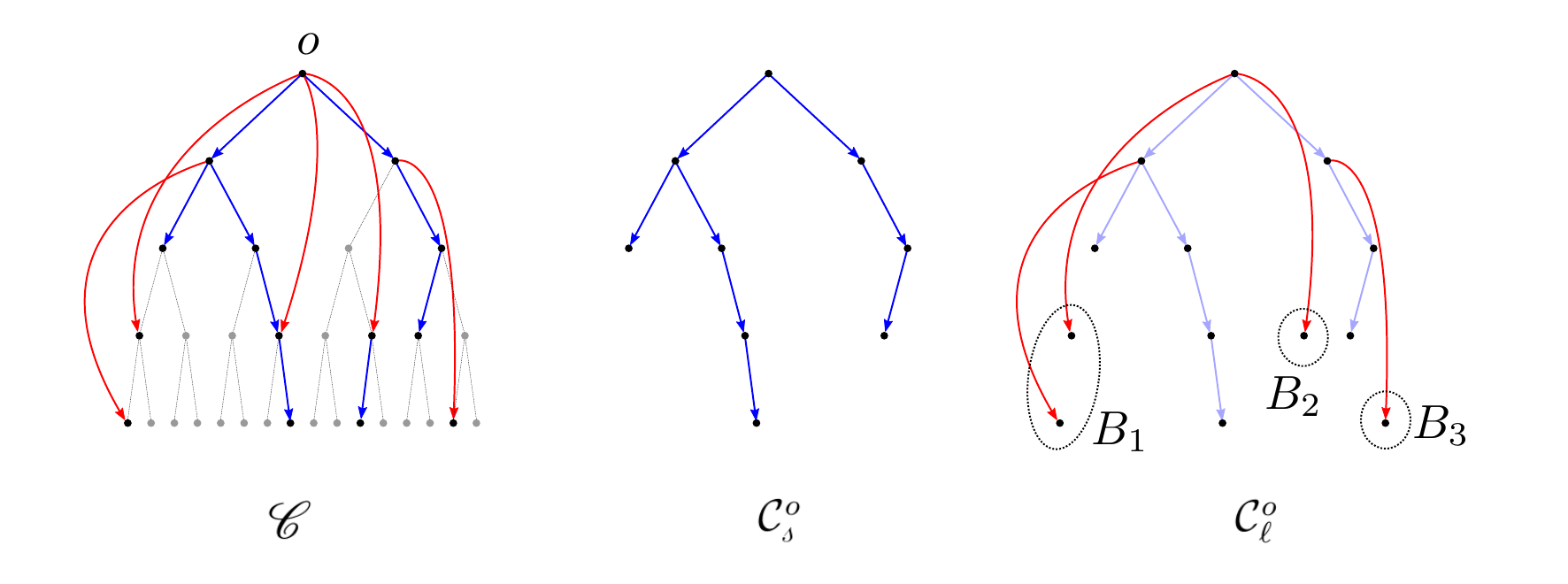}}
\end{center}
\caption{An example of $\mathscr{C}^o, \mathcal{C}_\s^o$ and~$\mathcal{C}_\ell^o$ on~$\mathbb{T}_{2, 3}$. In this case~$\mathcal{B}^o=\{B_1, B_2, B_3\}$ }
\label{fig:23pqregq}
\end{figure}

\begin{definition}
	Let~$B \subset V$ be admissible, with base~$u = \mathdutchcal{b}(B)$.  The \textit{type} of~$B$ is defined as $$\dt(B):= \{v \in V_\Gamma^{o}: u \cdot v \in B\},$$
that is,~$\dt(B)$ is the set of vertices which must be concatenated with~$u$ in order to produce~$B$. 
 The set~$\tpb$ of possible types is the set of all admissible sets with base equal to the  root~$o$,
	$$\tpb = \{A \subset V_\Gamma^o: o \in A\}.$$
\end{definition}
The set~$\tpb$ will be the space of types of the multi-type branching process that we will soon define. We emphasize that, although in Sections~\ref{sect:branchingproc} and~\ref{ss:further_mtbp} we represented types by lowercase letters~($a$,~$b$), here we will represent them by uppercase letters~($A$,~$B$), since they are themselves sets.

With the above definition at hand, we can now state:
\begin{corollary}\label{cor:spatial_markov}
	Let~$B \subset V$ be admissible. Condition on~$\mathcal{C}_\s^B$ and~$\mathcal{C}^B_\ell$ (so that~$\mathcal{B}^B$ is also determined). Then, the clusters
	$$\mathscr{C}^{B'}: B' \in \mathcal{B}^B$$
	are independent. Moreover, if~$B' \in \mathcal{B}^B$ with~$\mathdutchcal{b}(B') = u$ and~$\dt(B') = A$, then the distribution of
	$$\{v:u \cdot v \in \mathscr{C}^{B'}\}$$ 
	is equal to the (unconditioned) distribution of~$\mathscr{C}^{A}$.
\end{corollary}
\begin{proof}
	Fix~$U, U'\subset V$ for which the event~$\{ \mathcal{C}^B_\s = 
	U,\;\mathcal{C}^B_\ell = U' \}$ has positive probability, and condition on this event. Note that its occurrence can be decided by inspecting the statuses of all edges, short and long, that begin in~$U$. By Lemma~\ref{lem:decompositions}, none of these edges belong to any of the subtrees induced by the vertex sets
	$$V^{\mathdutchcal{b}(B')}:\;B' \in \mathcal{B}^B.$$
	The claimed independence then follows from the fact, also given in Lemma~\ref{lem:decompositions}, that these subtrees are disjoint. The claimed equality of distributions follows from the simple observation that the distribution of~$\{\omega(\langle u\cdot v,u \cdot v' \rangle): \langle v,v'\rangle \in E\}$ is equal to that of~$\{\omega(\langle v,v'  \rangle): \langle v,v'\rangle \in E\}$, for any~$u \in V$.
\end{proof}

We now apply Lemma~\ref{lem:decompositions} and Corollary~\ref{cor:spatial_markov} recursively to obtain important consequences. Letting~$B \subset V$ be admissible, define the collections~$\mathcal{Z}^B_0,\mathcal{Z}^B_1,\ldots$ of admissible sets recursively by letting~$\mathcal{Z}^B_0 = \{B\}$ and
$$\mathcal{Z}^B_{n+1} = \bigcup_{B'\in \mathcal{Z}^B_n} \mathcal{B}^{B'},\quad n \ge 0.$$
Then, Lemma~\ref{lem:decompositions} gives that  the sets 
$$V^{\mathdutchcal{b}(B')}: B' \in \mathcal{Z}^B_n,\;n \in \N_0$$
are pairwise disjoint. Moreover, we have the disjoint union
\begin{equation}\label{eq:cor_exp_bp}
\mathscr{C}^B = \bigcup_{n=0}^\infty \; \bigcup_{B' \in \mathcal{Z}^B_n}\mathcal{C}^{B'}_s.
\end{equation}
For each~$n \in \mathbb{N}_0$, we now define~${\mathbf{Z}}_n\in (\mathbb{N}_0)^{\tpb}$ by
$${\mathbf{Z}}_n(A) := |\{B' \in \mathcal{Z}^B_n: \dt(B') = A \}|, \quad A \in \tpb,$$
that is,~${\mathbf{Z}}_n(A)$ is the number of sets in~$\mathcal{Z}_n$ that have type~$A$. This gives a stochastic process~$\mathbf{Z}=({\mathbf{Z}}_n)_{n \ge 0}$ with state space~$\mathbb{N}_0^{\mathscr{A}}$, and as a consequence of Corollary~\ref{cor:spatial_markov} we obtain:
\begin{corollary}\label{cor:bp_app}
	The process~${\mathbf{Z}} = ({\mathbf{Z}}_n)_{n \ge 0}$ is a multi-type branching process with set of types equal to~$\mathscr{A}$, and initial population consisting of one individual with type~$\dt(B)$.
\end{corollary}
We observe that, due to our assumption that~$pd < 1$, almost surely all the short clusters that appear on the right-hand side of~\eqref{eq:cor_exp_bp} are finite. It then readily follows that
\begin{equation}
\label{eq:inf_iff}
\mathbb{P}(|\mathscr{C}^B| = \infty) = \mathbb{P}({\mathbf{Z}}\text{ survives}) = \zeta(B).
\end{equation}
It is easy to see that the mean offspring matrix~$M$ of~$\mathbf{Z}$, defined as in~\eqref{eq:matrix_M}, satisfies the conditions~\eqref{eq:not_rw} and~\eqref{eq:pos}, since regardless of~$B$, already in generation~1 all types can occur with positive probability.

\subsection{Analysis of mean offspring matrix}
Our interest is now in studying the behavior of~$q_c= q_c(k,p,d)$ as~$k\to \infty$. Specifically, we parametrize
\begin{equation}\label{eq:param_q}
q = \frac{1-pd}{d^k} + \frac{s}{d^{2k}},
\end{equation}
define
\begin{equation*}
\sst = \frac{(1-pd)^2p^2d}{1-p^2d},
\end{equation*}
and will prove that, on the one hand, there is no percolation when~$s < s_\star$ and~$k$ is large, and on the other hand, there is percolation when~$s > s_\star$ and~$k$ is large. Although most of the definitions we give below will depend on~$p,d,k$ and~$q$ (or~$s$), we will omit this dependence from the notation.

\begin{proposition} \label{prop:scomp} 
	If~$s > \sst$, then for~$k$ large enough we have
	\begin{equation}\label{eq:in_prop_a}
	M(\{o\},\{o\}) + \sum_{B:|B|=2} M(\{o\},B)\cdot M(B,\{o\}) > 1.
	\end{equation}
	If~$s < \sst$, then for~$k$ large enough we have
	\begin{equation} \label{eq:in_prop_b}
	\sum_{B:|B|\neq 2} M(\{o\},B) \cdot |B| + \sum_{B:|B| = 2} \sum_{B'} M(\{o\},B)\cdot M(B,B')\cdot |B'|  < 1. 
	\end{equation}
\end{proposition}
We postpone the proof of this result, and for now show how it almost immediately gives Theorem~\ref{thm:qck}.

\begin{proof}[{Proof of {Theorem~\ref{thm:qck}}.}] Recall from~\eqref{eq:inf_iff} that percolation occurs with positive probability if and only if~$\mathbf{Z}$ has positive probability of surviving.  The result then follows from combining Lemma~\ref{lem:main_comparison_bp} with Proposition~\ref{prop:scomp}.
\end{proof}

In the rest of this section, we prove Proposition~\ref{prop:scomp}. Several preliminary results will be needed. We start with some simple observations.
\begin{lemma}
If~$k$ is large enough, then for any~$A \in \tpb$ we have
\begin{equation}\label{eq:simple_bound}
\sum_{A'} M(A,A')\cdot |A'| \le 2|A|\quad \text{for all }A \in \tpb.
\end{equation}
\end{lemma}
\begin{proof}
For any~$A \in \tpb$ we have
\begin{equation}\label{eq:bound_all_sum}
\sum_{A'}{M}(A,A')\cdot |A'| = \mathbb{E}\left[|\mathcal{C}_\ell(A)|\right] \le qd^k\cdot \mathbb{E}\left[|\mathcal{C}_{\s}(A)|\right].
\end{equation}
Next, note that
$$\mathbb{E}\left[|\mathcal{C}_\s(A)|\right] \le |A|\cdot \mathbb{E}\left[|\mathcal{C}_\s(\{o\})|\right] = \frac{|A|}{1-pd}$$
Using~\eqref{eq:param_q}, the desired inequality then holds for~$k$ large enough that~$qd^k/(1-pd) < 2$,
\end{proof}

Next, fix~$A \in \tpb$. We observe that~$M(\{o\},A)$ is equal to the expected number of vertices~$v$ satisfying the following properties:
\begin{itemize}
\item[$\mathrm{(i)}$] $v \in \mathcal{C}_\ell(\{o\})$;
\item[$\mathrm{(ii)}$] $v$ is the unique vertex of~$\mathcal{C}_\ell(\{o\})$ in the path (of short edges) from~$o$ to~$v$;
\item[$\mathrm{(iii)}$] we have the equality of sets $$  \mathcal{C}_\ell(\{o\}) \cap V^v = \{v \cdot u:\; u \in A\}.$$
\end{itemize} 
We now define a quantity~$\barm(A)$ as the expected number of vertices satisfying a weaker list of requirements (specifically, we remove requirement~$\mathrm{(ii)}$ and replace the set equality in~$\mathrm{(iii)}$ by a set inclusion):
\begin{definition}
Given~$A \in \tpb$, let~$\barm(A)$ be the expected number of vertices~$v \in \mathcal{C}_\ell(\{o\})$ such that
$$\mathcal{C}_\ell(\{o\}) \cap V^v \supseteq \{v\cdot u:u\in A\}.$$
\end{definition}
Note in particular that for~$A = \{o\}$, we have
\begin{equation*}\barm(\{o\}) = \mathbb{E}[|\mathcal{C}_\ell|].\label{eq:card_ell}\end{equation*}
Also note that~$M(\{o\},A) \le \barm(A)$ for all~$A$, and a moment's thought about the two quantities leads to the equality\begin{equation}\label{eq:first_barm}
M(\{o\},\{o\}) = \barm(\{o\}) - \sum_{A:|A| \ge 2} M(\{o\},A)\cdot |A|.
\end{equation}

We now state three lemmas that we will need for the proof of Proposition~\ref{prop:scomp}.

\begin{lemma}[Formula for~$\bar{M}$]\label{lem:use_height}
	For any~$A \in \tpb$ we have
	\begin{equation}\label{eq:barm_bound}
	\barm(A) =  (1-p^{k-h(A)})\cdot \frac{d^k\cdot q^{|A|} \cdot p^{h(A)}}{1-pd}.
	\end{equation}
In particular, we have
\begin{equation} \label{eq:particular_o}
d^k\cdot (\bar{M}(\{o\}) - 1) \xrightarrow{k \to \infty} \frac{s}{1-pd}
\end{equation}
and, for~$A \in \tpb$ with~$|A| = 2$,
\begin{equation}\label{eq:for_ma1}\lim_{k \to \infty} d^k\cdot \barm(A) = (1-pd)\cdot p^{h(A)}.\end{equation}
\end{lemma}

\begin{lemma}[Estimate for types with cardinality above~2]\label{lem:three_more}
	We have
	\begin{equation}\label{eq:boundthree}
	d^k\sum_{A:|A|\ge 3} \barm(A) \cdot |A|^2 \xrightarrow{k \to \infty} 0.
	\end{equation}
\end{lemma}

\begin{lemma}[Estimate for types with cardinality~2] \label{lem:card2} Let~$A \in \tpb$ with~$|A| = 2$. We have
	\begin{equation}\label{eq:conv_of_two}
	d^k \cdot M(\{o\},A) \xrightarrow{k \to\infty} (1-pd)\cdot p^{h(A)}.
	\end{equation}
Moreover,
\begin{equation}
\label{eq:upper_bound_card2}
\sum_{A'}M(A,A')\cdot |A'| \le \frac{q\cdot d^k \cdot (2-p^{h(A)})}{1-pd}
\end{equation}
and
\begin{equation}\label{eq:conv_from_two}
	M(A,\{o\}) \xrightarrow{k \to \infty} 2-p^{h(A)}
	\end{equation}
\end{lemma}
We postpone the proofs of these three lemmas, and for now we show how they can be combined to prove Proposition~\ref{prop:scomp}.

\begin{proof}[Proof of {{Proposition~\ref{prop:scomp}}}]
Using~\eqref{eq:first_barm} and~\eqref{eq:particular_o} we have
\[M(\{o\},\{o\}) = 1 + \frac{s}{(1-pd)d^k} - \sum_{A:|A| \ge 2}M(\{o\},A)\cdot |A| + o \left(\frac{1}{d^k}\right).\]
Combining this with the inequality~$M(\{o\},A) \le \barm(A)$ and Lemma~\ref{lem:three_more}, we obtain
	\begin{equation}\label{eq:int_ek} M(\{o\},\{o\}) = 1 + \frac{s}{(1-pd)d^k} - 2\sum_{A:|A| = 2} M(\{o\},A) + o\left(\frac{1}{d^k} \right). \end{equation}

	With this at hand, the left-hand side of~\eqref{eq:in_prop_a} can be written
	\begin{equation}\label{eq:square_br}
	1+\frac{1}{d^k} \left[\frac{s}{1-pd} + \sum_{A:|A| = 2} d^k \cdot M(\{o\},A) \cdot (M(A,\{o\})-2) + o(1)\right].
	\end{equation}
	By~\eqref{eq:conv_of_two} and~\eqref{eq:conv_from_two}, for any~$A$ with~$|A|=2$ we have
	$$d^k \cdot M(\{o\},A) \cdot (M(A,\{o\})-2) \xrightarrow{k \to \infty}-(1-pd)\cdot p^{2h(A)}.$$
	Moreover, using~$M(\{o\},A) \le \bar{M}(A)$ again, together with~\eqref{eq:simple_bound} and~\eqref{eq:barm_bound}, we have the domination
	\begin{align*}&\sum_{A:|A|=2} \left| d^k \cdot M(\{o\},A)\cdot (M(A,\{o\})-2)\right|\\&\le\sum_{A:|A| = 2} d^k \cdot \frac{d^k\cdot q^2\cdot  p^{h(A)}}{1-pd}\cdot  (2|A|-2) = \frac{2d^{2k}}{1-pd}\cdot \left(\frac{1-pd}{d^k} + \frac{s}{d^{2k}}\right)^2 \cdot \sum_{h=1}^{k-1} (pd)^h, \end{align*}
	which is bounded from above uniformly in~$k$, by the choice of~$q$ {in \eqref{eq:param_q}} and~$pd < 1$. Hence, by dominated convergence, the term inside the square brackets in~\eqref{eq:square_br} converges to
	\begin{equation}\label{eq:lim_sqbk}\frac{s}{1-pd} + \sum_{h=1}^\infty d^h\cdot \left(-(1-pd)\cdot p^{2h}\right) = \frac{s}{1-pd} - (1-pd)\cdot \frac{p^2d}{1-p^2d}.\end{equation}
	This is strictly positive when~$s > \sst$, so the expression in~\eqref{eq:square_br} is larger than one when~$s > \sst$ and~$k$ is large enough.
	
	Next, by Lemma~\ref{lem:three_more}, the left-hand side of~\eqref{eq:in_prop_b} is
\[ M(\{o\},\{o\}) + \sum_{B:|B| = 2} \sum_{B'} M(\{o\},B)\cdot M(B,B')\cdot |B'| + \frac{1}{d^k}\cdot o(1). \]
Again using~\eqref{eq:int_ek}, this is equal to
$$
	1+\frac{1}{d^k} \left[\frac{s}{1-pd} + \sum_{A:|A| = 2} d^k \cdot M(\{o\},A) \cdot \left(-2+ \sum_{A'} M(A,A')\cdot |A'|\right) + o(1)\right].$$
Using~\eqref{eq:upper_bound_card2}, this is bounded from above by
\begin{align}\label{eq:helpful_sqb}
 1+\frac{1}{d^k} \left[\frac{s}{1-pd} + \sum_{A:|A| = 2} d^k \cdot M(\{o\},A) \cdot \left(-2+ \frac{q\cdot d^k\cdot (2-p^{h(A)})}{1-pd}\right) + o(1)\right].
	\end{align}
Using the fact that~$q = (1-pd)d^{-k} + sd^{-2k}$ and~\eqref{eq:conv_of_two}, we have, for any~$A$ with~$|A|=2$:
$$d^k \cdot M(\{o\},A) \cdot \left(-2+\frac{q\cdot d^k\cdot (2-p^{h(A)})}{1-pd}\right) \xrightarrow{k \to \infty} - (1-pd)\cdot p^{2h(A)}. $$
	Arguing using dominated convergence as before, the term inside the square brackets in~\eqref{eq:helpful_sqb} converges to the same limit as in~\eqref{eq:lim_sqbk}, which is negative when~$s<\sst$. This completes the proof.
\end{proof}

We now turn to the proofs of the three lemmas.

\begin{proof}[Proof of {{Lemma~\ref{lem:use_height}}}]
	Fix~$A \in \tpb$ and~$v \in V$ with~$h(v) \ge k$. Write~$v = (v_1,\ldots,v_n)$, with~$n = h(v)$, and let~$v_{(i)}:= (v_1,\ldots, v_i)$ for~$1 \le i \le n$. The event that~$\{v\cdot u: u\in A\} \subset \mathcal{C}_\ell(\{o\})$ occurs if and only if:
	\begin{enumerate}
		\item vertices~$v_{(1)},\ldots, v_{(n-k+h(A))}$ are all in~$\mathcal{C}_\s(\{o\})$;
		\item vertex~$v$ is not in~$\mathcal{C}_{\s}(\{o\})$;
		\item all the long edges starting from the set~$\{o,v_{(1)},\ldots, v_{(n-k+h(A))}\}$ and ending at the set~$\{v\cdot u:u \in A\}$ are open.
	\end{enumerate}
	This has probability~$p^{h(v)+h(A)-k}\cdot (1-p^{k-h(A)})\cdot q^{|A|}$. The result follows from summing this over all~$v$ with~$h(v) \ge k$:
\begin{align*}
&\sum_{v:h(v)\ge k}p^{h(v)+h(A)-k}\cdot (1-p^{k-h(A)})\cdot q^{|A|} \\&= (1-p^{k-h(A)})\cdot q^{|A|}\cdot p^{h(A)-k}\cdot \sum_{h =k}^\infty d^hp^h =  (1-p^{k-h(A)})\cdot \frac{d^k\cdot q^{|A|}\cdot  p^{h(A)}}{1-pd}.
\end{align*}
This proves~\eqref{eq:barm_bound}.

For~\eqref{eq:particular_o}, we combine~\eqref{eq:barm_bound} with the equality~$q = (1-pd)d^{-k}+ sd^{-2k}$ to get
\begin{align*}
d^k\cdot (\barm(\{o\}) - 1) &= d^k \cdot \left((1-p^k)\cdot \frac{d^k}{1-pd}\cdot \left(\frac{1-pd}{d^k}+\frac{s}{d^{2k}}\right)-1 \right)\\
&= d^k \cdot \left(\frac{s}{d^k\cdot (1-pd)} - p^k - \frac{s\cdot p^k}{d^k\cdot (1-pd)}\right) \xrightarrow{k \to \infty} \frac{s}{1-pd},
\end{align*}
since~$p^k = o(1/d^k)$.

Similarly, for~\eqref{eq:for_ma1}, for~$|A| = 2$ we use~$d^{2k}\cdot q^2 \to (1-pd)^2$ to obtain
$$d^k \cdot \barm(A) = d^k \cdot (1-p^{k-h(A)}) \cdot \frac{d^k \cdot q^2\cdot p^{h(A)}}{1-pd} \xrightarrow{k \to \infty} (1-pd)\cdot p^{h(A)}.$$
\end{proof}

\begin{proof}[{Proof of {Lemma~\ref{lem:three_more}}} ]
	For~$m,h \in \mathbb{N}$, let~$\eta(m,h)$ denote the number of sets~$A \in \tpb$ with~$|A| = m$ and~$h(A) = h$. Keeping in mind that any~$A \in \tpb$ contains~$\{o\}$, we bound
	$$\eta(m,h) \le {d+\cdots +d^h \choose m-1} \le \frac{(d+\cdots + d^h)^{m-1}}{(m-1)!} \le \frac{d^{(h+1)(m-1)}}{(m-1)!},$$
	since~$(d^{h+1}-1)/(d-1) \le d^{h+1}$. Then, for each~$m \in \mathbb{N}$,
	\begin{align*}
	\sum_{h=1}^{k-1}\; \sum_{\substack{A:|A|=m,\\h(A)=h}} \barm(A) &\stackrel{\eqref{eq:barm_bound}}{\le} \sum_{h=1}^{k-1} \eta(m,h) \cdot \frac{d^k q^m p^h}{1-pd} \le \frac{d^kq^m\sum_{h=1}^{k-1} (pd^{m-1})^{h+1}}{p(1-pd) (m-1)!}. 
	\end{align*}
	We can bound the sum in the numerator by~$k-1$ in case~$pd^{m-1} \le 1$ and by~$(k-1)(pd^{m-1})^{k}$ otherwise, and then, bounding a maximum by a sum,
	$$\sum_{h=1}^{k-1} (pd^{m-1})^{h+1} \le (k-1)\cdot(1+(pd^{m-1})^k).$$
	We then obtain
	\begin{align*}
&	d^k \sum_{A:|A| \ge 3}\barm(A) \cdot |A|^2 \le \frac{(k-1)d^{2k}}{p(1-pd)} \sum_{m \ge 3} \frac{m^2 q^m}{ (m-1)!} \cdot \left(1 + p^kd^{k(m-1)}\right)\\[.2cm]
	&\le \frac{(k-1)d^{2k}q^3}{p(1-pd)}\sum_{m\ge 3}\frac{m^2}{(m-1)!} + \frac{(k-1)d^{2k}qp^k}{p(1-pd)}\sum_{m\ge 3}\frac{m^2}{(m-1)!}\cdot (qd^k)^{m-1}. 
	\end{align*}
	Using~$q = O\left(\frac{1}{d^k}\right)$,~$p^k = o\left(\frac{1}{d^k}\right)$ and~$\sum_{m\ge 3}\frac{m^2}{(m-1)!} < \infty$, we see that this tends to zero as~$k \to \infty$.
\end{proof}

\begin{proof}[{Proof of {Lemma~\ref{lem:card2}}}] Fix~$A$ with~$|A| = 2$. We note that, similarly to~\eqref{eq:first_barm}, we have
\begin{equation}\label{eq:for_ma2}\barm(A) \ge M(\{o\},A) \ge \barm(A) - \sum_{A':|A'| \ge 3} {|A'| \choose 2}\cdot M(\{o\},A'). \end{equation}
	The convergence~\eqref{eq:conv_of_two} now follows from combining this with~\eqref{eq:for_ma1} and~\eqref{eq:boundthree}.

To prove the remaining statements, we write~$A = \{o,u\}$ and compute
	\begin{align*}
	\mathbb{E}\left[|\mathcal{C}_\s(A)|\right] &= \mathbb{E}\left[|\mathcal{C}_\s(\{o\})|\right] + \mathbb{E}\left[|\mathcal{C}_\s(\{u\})|\right] - \sum_{v \in V^u}\mathbb{P}\left(v \in \mathcal{C}_\s(\{o\})\right) \\&= \frac{2-p^{h(u)}}{1-pd} =\frac{2-p^{h(A)}}{1-pd}.
	\end{align*}
We then have
$$M(A,\{o\}) \le\sum_{A'} M(A,A')\cdot |A'| \le q d^k\cdot \E[|\mathcal{C}_\s(A)|] = q d^k \cdot \frac{2-p^{h(A)}}{1-pd} .$$
This already proves~\eqref{eq:upper_bound_card2}. For~\eqref{eq:conv_from_two}, we need two additional statements. First,
$$qd^k \cdot \E[|\mathcal{C}_\s(A)|]  \xrightarrow{k \to \infty} 2-p^{h(A)}.$$
This is an immediate consequence of~$q = (1-pd)d^{-k} + 2d^{-2k}$. Second,
$$qd^k \cdot \E[|\mathcal{C}_\s(A)|] - M(A,\{o\}) \xrightarrow{k \to \infty} 0.$$
This is proved by  an argument similar to the proof of Lemma~\ref{lem:three_more}: the expected number of open long edges that start at~$\mathcal{C}_\s(A)$, but do not contribute to the value of~$\mathbf{Z}_1({\{o\}})$ (either because they also end at~$\mathcal{C}_\s(A)$, or because they  end at a vertex belonging to a set of~$\mathcal{B}^A$ which is not of type~$\{o\}$) tends to zero as~$k \to \infty$. We omit the details. 
\end{proof}

\section{Limit theorems}\label{sect:limitthm}
In this section, we prove Theorem~\ref{thm:limits}. We will first recall some limit theorems from the literature of multi-type branching processes. We will then give a second branching process representation of the percolation cluster, and readily obtain the results for the three regimes as stated in Theorem~\ref{thm:limits}.
\subsection{Limit theorems  for multi-type branching processes} \label{ss:limit_theorems}
 As in Section~\ref{sect:branchingproc}, we consider a multi-type branching process~$(\mathbf{Z}_n)_{n \ge 0}$ with set of types~$\mathscr{T}$, offspring distribution~$\p$ and mean offspring matrix~$M$ with Perron-Frobenius eigenvalue~$\rho$. We continue assuming that conditions~\eqref{eq:not_rw} and~\eqref{eq:pos} hold.
 We also let~$\mu,\nu \in (\N_0)^\mathscr{T}$ respectively denote left and right eigenvectors of~$M$ corresponding to~$\rho$, that is,
\begin{equation}\label{eq:evs}\sum_{a \in \mathscr{T}} \mu(a)\cdot M(a,b) = \rho \cdot \mu(b)\;\; \forall b \in \mathscr{T},\qquad \sum_{b \in \mathscr{T}} M(a,b) \cdot \nu(b) = \rho \cdot \nu(a) \;\; \forall a \in \mathscr{T};\end{equation}
these eigenvectors are uniquely determined if we require the normalization conditions
\begin{equation}\label{eq:normalization}
\sum_{a \in \mathscr{T}} \mu(a) = 1,\qquad \sum_{a \in \mathscr{T}} \mu(a) \cdot \nu(a) = 1.
\end{equation}
Moreover, it follows from the Perron-Frobenius Theorem and condition~\eqref{eq:pos} that~$\mu(a) > 0$ and~$\nu(a) > 0$ for all~$a \in \mathscr{T}$.

The following result  is contained in Theorem~1 in Chapter~V, Section 6 of \cite{AN72}, concerning the supercritical regime:
\begin{theorem}\label{thm:supercritAN}
Assume that~$\rho > 1$ and 
\[
 \sum_\eta \eta(b) \cdot \log(\eta(b))\cdot \p(a,\eta)<\infty \text{ for all } a, b\in \mathscr{T}. 
\]
We then have
	\begin{equation}\label{eq:log}
	\lim_{n\rightarrow\infty}\frac{\mathbf{Z}_n}{\rho^n}=Z \cdot \nu \text{ almost surely},
	\end{equation}
	where~$Z$ is a nonnegative random variable such that~$\P(Z>0)>0$.
\end{theorem}
For the subcritical regime, the following statement is contained in Theorem~2 in Chapter~V, Section 4 of~\cite{AN72}:
\begin{theorem}\label{thm:subcritAN}
	If~$\rho < 1$, then for each~$a \in \mathscr{T}$ and~$\eta \in (\N_0)^\mathscr{T}$, the limit
	\[
Q(\eta) :=	\lim_{n\rightarrow\infty}\P(\mathbf{Z}_n=\eta\mid\mathbf{Z}_0=\mathbf{e}_a,\; \mathbf{Z}_n\neq\bold{0})
	\]
	exists and is independent of~$a$. Moreover,~$\sum_\eta Q(\eta) = 1$.
\end{theorem}

For the result we will reproduce concerning the critical case~$\rho =1$, we first introduce some terminology, starting with convergence in distribution (in the sense of Benjamini and Schramm~\cite{BS01}) for sequences of random rooted trees. A rooted tree~$\dot{\tau} = (\tau,\mathcal{O})$ is a tree~$\tau$ with a distinguished vertex~$\mathcal{O}$. For each~$m \in \N_0$, let us denote by~$B_m(\dot{\tau})$ the rooted tree obtained as the subgraph of~$\tau$ induced by all vertices at graph distance at most~$m$ from~$\mathcal{O}$, rooted again at~$\mathcal{O}$. We say that a sequence of random  rooted trees~$\dot{\tau}_n$ converges in distribution to a random rooted tree~$\dot{\tau}$ if, for any~$m \in \N_0$ and any deterministic finite rooted tree~$(\bar{t},\bar{o})$, we have
\[\P(B_m(\dot{\tau}_n) \cong (\bar{t},\bar{o})) \xrightarrow{n \to \infty} \P(B_m(\dot{\tau}) \cong (\bar{t},\bar{o})),\]
where~`$\cong$' denotes graph isomorphism between rooted trees.

We now recall some facts about convergence in distribution of family trees of critical branching processes. In~\cite{K86}, Kesten introduced a random tree with an infinite spine that is the local limit in distribution of the family tree of a critical (single-type) branching process conditioned on reaching generation~$n$, as~$n\rightarrow\infty$.
In \cite{AD14} the authors gave a necessary and sufficient condition for the convergence in distribution of a family tree to the corresponding Kesten's tree.
Later this result was generalized for multi-type branching processes in~\cite{ADG18} and~\cite{S18}. The following follows from Theorem~3.1 in~\cite{S18}:
\begin{theorem}\label{thm:critical_mtbp}
Assume that~$\rho  = 1$ and that there exists~$\theta > 0$ such that
\[\sum_{\eta \in (\N_0)^\mathscr{T}} \p(a,\eta)\cdot \exp\left\{\theta \cdot \sum_{b \in \mathscr{T}} \eta(b)\right\} < \infty \quad \text{for all } a \in \mathscr{T}.\]
Then, the family tree of~$(\mathbf{Z}_n)_{n \ge 0}$ conditioned on having more than~$n$ vertices converges in distribution, as~$n \to \infty$, to a random rooted tree.
\end{theorem}
In fact, in~\cite{S18} a description of the limiting tree is provided: it is again a Galton-Watson tree with an augmented space of types and an offspring distribution that is obtained from~$\p$ and the eigenvalues associated to~$\rho = 1$. We refrain from recalling the details here, since we will not use them.

\subsection{Second branching process representation of percolation cluster}
We again go back to our percolation model on~$\mathbb{T}$ with parameters~$p,q \in [0,1]^2$. Recall that~$V^o_\Gamma = \{u \in V:\; 0 \le h(u)< k\}$.
\begin{definition}\label{def:s_and_w} Define the random sets
$$S(v) :=  \{u \in V^o_\Gamma:\; v\cdot u \in \mathscr{C}\},\qquad v \in V.$$
Let~$\mathscr{B}$ denote the collection of all non-empty subsets of~$V^o_\Gamma$ and define the process~$\mathbf{W} = (\mathbf{W}_n)_{n \ge 0}$ on~$(\N_0)^\mathscr{B}$ by setting
\begin{equation*}
\mathbf{W}_n(A) := |\{v:\;h(v) = n,\; S(v) = A\}|,\qquad A \in \mathscr{B},\; n \in \N_0.
\end{equation*}
\end{definition}
By considering the exploration process~$(\mathscr{C} \cap \{v:h(v)\le k-1+n\})_{n \ge 0}$, it is readily seen that~$\mathbf{W}$ is a multi-type branching process with set of types~$\mathscr{B}$. The initial population is given by a single individual of a random type, with distribution equal to that of~$\mathscr{C} \cap V^o_\Gamma$. The mean offspring matrix of~$\mathbf{W}$ will (again) be denoted~$M$. We will write~$M_{p,q}$ when we want to  make the parameters explicit. It is not difficult to check that~$M$ satisfies the conditions~\eqref{eq:not_rw} and~\eqref{eq:pos} given in Section~\ref{sect:branchingproc}; we leave this to the reader. We again let~$\rho=\rho_{p,q}$ denote the Perron-Frobenius eigenvalue of~$M_{p,q}$.
\begin{lemma}\label{lem:corresp}
We have
\[\rho_{p,q} \begin{cases}
>1 &\text{if } q > q_c(p);\\
=1 &\text{if } q = q_c(p);\\
<1 &\text{if } q < q_c(p).
\end{cases}\]
\end{lemma}
\begin{proof}
Note that~$\mathbf{W}$ survives  if and only if~$\mathscr{C}$ is infinite. Moreover, as mentioned earlier, it was proved in~\cite{LRV17} that~$\mathscr{C}$ is almost surely finite when~$q = q_c(p)$. This implies that~$\rho_{p,q} > 1$ if and only if~$q > q_c(p)$. Also note that the entries of~$M_{p,q}$ depend continuously on~$(p,q)$. Since~$\rho_{p,q}$ is a simple root of the characteristic equation of this matrix, it depends continuously (in fact smoothly) on the entries of the matrix, hence~$\rho_{p,q}$ also depends continuously on~$(p,q)$. This implies that~$\rho_{p,q_c(p)} = 1$.

It remains to prove that~$\rho_{p,q} < 1$ when~$q < q_c(p)$. To this end, fix~$(p,q)$ with~$q < q_c(p)$. Also fix~$q' \in (q,q_c(p))$. We construct a standard coupling of the clusters~$\mathscr{C}_{p,q}$ and~$\mathscr{C}_{p,q'}$ using independent Uniform$([0,1])$ random variables indexed by the edges of~$\mathbb{T}$,~$\{U_e: e \in E \}$. For both~$\mathscr{C}_{p,q}$ and~$\mathscr{C}_{p,q'}$, we declare a short edge~$e$ to be open if~$U_e \le p$. We declare a long edge~$e$ to be open for~$\mathscr{C}_{p,q}$ if~$U_e \le q$, and we declare it to be open for~$\mathscr{C}_{p,q'}$ if~$U_e \le q'$. Now, in this same probability space, we define an additional cluster~${ \mathscr{C}}^*$, which will satisfy~$\mathscr{C}_{p,q}\subset \mathscr{C}^* \subset \mathscr{C}_{p,q'}$, and will give rise to a multi-type branching process~$\mathbf{W}^*$ with Perron-Frobenius eigenvalue~$\rho^*$. We will then show that~$\rho_{p,q} < \rho^* \le 1$, completing the proof.

We will define~$\mathscr{C}^*$ by recursively defining~$\mathscr{C}^* \cap \{v:h(v) \le n\}$ for each~$n \ge k-1$. Start the recursion by setting
$$\mathscr{C}^* \cap \{v:h(v)\le k-1\} = \mathscr{C}_{p,q} \cap \{v:h(v)\le k-1\}.$$ Now assume that~$\mathscr{C}^* \cap \{v:h(v) \le n\}$ has been defined. Fix~$v$ with~$h(v) = n+1$; let~$v'$ be the parent of~$v$ (so that~$\langle v',v\rangle$ is a short edge of~$\mathbb{T}$) and let~$v''$ be the ancestor of~$v$ at distance~$k$ from~$v$ (so that~$\langle v'',v\rangle$ is a long edge of~$\mathbb{T}$).
 We include~$v$ in~$\mathscr{C}^*$ in any of the following three situations:
\begin{itemize}
\item[(a)] $v' \in \mathscr{C}^*$, and~$U_{\langle v', v \rangle} \le p$;
\item[(b)] $v'' \in \mathscr{C}^*$, and~$U_{\langle v'',v\rangle} \le q$;
\item[(c)] $v'' \in \mathscr{C}^*$,~$U_{\langle v'', v \rangle} \in (q,q']$, and~$V^{v''}_\Gamma \cap \mathscr{C}^* = \{v''\}$ (that is, no descendant of~$v''$ with height between~$h(v'')+1$ and~$h(v)-1$ has been included in~$\mathscr{C}^*$).
\end{itemize}
Put simply, in an exploration that adds one height unit at a time, the cluster~$\mathscr{C}^*$ grows in the same way as~$\mathscr{C}_{p,q}$, with the exception of rule~(c) above, which prescribes that vertices that would otherwise become leaves are given extra chances (with probability~$q'-q$ each) of having (long-distance) neighbors.

We now replicate Definition~\ref{def:s_and_w} by letting
$$S^*(v):= \{u \in V^o_\Gamma: \; v\cdot u \in \mathscr{C}^*\},\qquad v \in V, $$
and defining~$\mathbf{W}^* = (\mathbf{W}_n^*)_{n \ge 0}$ by 
$$\mathbf{W}_n^*(A):= |\{v:\;h(v) = n,\; S^*(v) = A\}|,\qquad A \in \mathscr{B},\; n \in \mathbb{N}_0.$$
Then,~$\mathbf{W}^*$ is also a multi-type branching process. Its mean offspring matrix, denoted~$M^*$, satisfies:
\begin{equation}\label{eq:pf_n*} \begin{split}
&M^*(A,B) = M_{p,q}(A,B) \text{ if }(A,B) \neq (\{o\},\{o\}),\\ & M^*(\{o\},\{o\}) > M_{p,q}(\{o\},\{o\}).\end{split}
\end{equation}
We have that
$$\mathbb{P}(\mathbf{W}^* \text{ survives}) =\mathbb{P}(|\mathscr{C}^*| = \infty) \le \mathbb{P}(|\mathscr{C}_{p,q'}| = \infty) = 0,$$
so the Perron-Frobenius eigenvalue~$\rho^*$ of~$M^*$ is at most~$1$. We will now show that~\eqref{eq:pf_n*} implies  that~$\rho^* > \rho_{p,q}$. Let~$\mu,\nu: \mathscr{B} \to \mathbb{R}$  be left and right eigenvectors of~$M_{p,q}$ associated to the eigenvalue~$\rho_{p,q}$, as in~\eqref{eq:evs} and~\eqref{eq:normalization}.
Recall that all entries of~$\mu$ and~$\nu$ are strictly positive. Fix any norm~$\|\cdot \|$ on the space of matrices. The perturbation theory of Perron-Frobenius eigenvalues (see Theorem~8 in Chapter~8 of~\cite{MN19}) gives that the Perron-Frobenius eigenvalue of~$M_{p,q} +E$, where~$E$ is a matrix of sufficiently small norm, is given by
$$\rho(M_{p,q}+ E) = \rho_{p,q} + \frac{\sum_{A,B}\mu(A)\cdot E(A,B)\cdot \nu(B)}{\sum_A \mu(A) \cdot \nu(A)} + O(\|E\|^2).$$
In particular, if~$E$ is a non-negative matrix with at least one strictly positive entry, then~$\rho(M_{p,q} + E) > \rho_{p,q}$. The result is thus obtained  by taking~$E = M^*- M_{p,q}$.
\end{proof}

\begin{proof}[Proof of {{Theorem~\ref{thm:limits}}}]
Recall that~$X_n$ denotes the number of vertices~$u \in \mathscr{C}_{p,q}$ with~$h(u) = n$. We extract~$X_n$ from~$\mathbf{W}_n$ by writing
$$X_n = \sum_{B \in \mathscr{B}}\mathbf{W}_n(B) \cdot \mathds{1}\{o \in B\}.$$

Assume~$q > q_c(p)$. Then, by Lemma~\ref{lem:corresp} we have~$\rho > 1$, so Theorem~\ref{thm:supercritAN} gives~$\rho^{-n}\cdot \mathbf{W}_n \xrightarrow{n \to \infty} Z \cdot \nu$ for some random variable~$Z$ with~$\P(Z > 0) > 0$. We then obtain
$$\frac{X_n}{\rho^n} \xrightarrow{n \to \infty} Z\cdot \sum_{B \in \mathscr{B}} \nu(B)\cdot \mathds{1}\{o \in B\},$$
proving the first statement of the theorem. The statement for~$q < q_c(p)$ is proved similarly.

We would like to obtain the final statement, concerning the critical case, by applying Theorem~\ref{thm:critical_mtbp} to the family tree of~$\mathbf{W}$. However, there is a problem: the cluster~$\mathscr{C}$ cannot be recovered from the family tree of~$\mathbf{W}$. For instance, assume that~$\mathbf{W}_0$ is the population consisting of a single individual of type~$\{o\}$, and~$\mathbf{W}_1$ is the population consisting of a single individual of type~$\{v\}$, where~$v$ is a vertex with~$h(v) = k-1$ (so that~$v = (v_1,\ldots, v_{k-1})$, with~$v_1,\ldots,v_{k-1} \in [d]$). Then,~$\mathscr{C} \cap \{u:h(u) \le k\}$ could be any of the sets
\[\{o,(1,v_1,\ldots,v_{k-1})\},\;\{o,(2,v_1,\ldots,v_{k-1})\},\ldots, \{o,(d,v_1,\ldots, v_{k-1})\}. \]
To solve this issue, we modify the process~$\mathbf{W}$ slightly, by augmenting its set of types in order to encode the information that is missing, so that the correspondence between realizations of its family tree and realizations of~$\mathscr{C}$ becomes one-to-one.

To this end, first define the function~$f: V \to [d]$ by defining~$f(o)$ arbitrarily (say,~$f(o) = 1$), and letting
$$f(v_1,\ldots, v_n) = v_n \quad \text{for any } v = (v_1,\ldots, v_n) \in V \backslash \{o\}.$$ 
We then let~$\hat{\mathscr{B}} := \mathscr{B} \times [d]$ and define the process~$\hat{\mathbf{W}} = (\hat{\mathbf{W}}_n)_{n \ge 0}$ with space of types~$\hat{\mathscr{B}}$ by letting
$$\hat{\mathbf{W}}_n((A,i)) := |\{v:\;h(v) =n,\; S(v)=A,\; f(v) =i\}|,\; A \in \mathscr{B},\;i \in [d],\; n \in \N_0.$$
It should now be clear that~$\hat{\mathbf{W}}$ is a multi-type branching process, and moreover there is a natural bijection between~$\mathscr{C}$ and the trajectory~$(\hat{\mathbf{W}}_n)_{n \ge 0}$, so that (say) an individual of type~$(A,i)$ in generation~$n$ of~$\hat{\mathbf{W}}$ can be associated to a unique vertex~$v \in V$ with~$h(v) = n$ and~$S(v) = A$. Moreover, exactly the same proof as that of Lemma~\ref{lem:corresp} shows that the Perron-Frobenius eigenvalue~$\hat{\rho}$ of~$\hat{\mathbf{W}}$ equals~1 when~$q =q_c(p)$. The desired result now follows from applying Theorem~\ref{thm:critical_mtbp} to~$\hat{\mathbf{W}}$.
\end{proof}

\section{The critical curve}\label{sect:critcurve}
We finally prove Theorem~\ref{THM:CRITCURVE} in this section. We start giving a third branching process construction based on an exploration of the percolation cluster. This construction is well-suited to compare percolation on~$\mathbb{T}$ with a branching process whose offspring distribution is a sum of independent binomial random variables. Next, we use a coupling technique (Lemma~\ref{lemma:coupling} below) to argue for stochastic domination, even if the parameter value of the branching process is decreased slightly.
\subsection{Third branching process representation of percolation cluster}
In this section we prove Theorem \ref{THM:CRITCURVE}.
Let~$\widehat{\mathbb{T}}  = \widehat{\mathbb{T}}_{d,k}= (\widehat{V}, \widehat{E})$ be an oriented graph
with vertex and edge set
\begin{align*}
&\widehat{V}=[d+d^k]_*, \qquad \widehat{E}=\{\langle \hat r, \hat r \cdot \hat i \rangle: \hat r\in \hat{V},\; \hat i\in [d+d^k]\}.
\end{align*}
Thus in~$\widehat{\mathbb{T}}$ every vertex has out-degree~$d+d^k$. Denote the root by~$\hat o$.
Fix an arbitrary bijective function~$\varphi: [d+d^k] \to [d]\cup [d]^k$ and partition the edge set into subsets~$\widehat{E}_{\s}$ and~$\widehat{E}_{\ell}$ with
\begin{align*}
\widehat{E}_{\s}&=\{\langle \hat r, \hat r \cdot \hat i \rangle \in \widehat{E}: \varphi(\hat i)\in[d]\},\\
\widehat{E}_{\ell}&=\{\langle \hat r, \hat r \cdot \hat i \rangle \in \widehat{E}: \varphi(\hat i)\in[d]^k\}.
\end{align*}
Consider the following percolation model on~$\widehat{\mathbb{T}}$:
every edge in~$\widehat E_{\s}$ is open with probability~$p$, and every edge in~$\widehat{E}_{\ell}$ is open with probability~$q$.
Denote the cluster of~$\hat o$ by~$\widehat{\mathscr{C}} = \widehat{\mathscr{C}}_{p,q}$.
Note that~$\widehat{\mathscr{C}}$ has the same distribution as the family tree of the branching process with offspring distribution that is the sum of two binomial random variables,~$\text{Bin}(d, p)$ and~$\text{Bin}(d^k, q)$.
Theorem~\ref{THM:CRITCURVE} is a direct consequence of the following proposition.

\begin{proposition}\label{prop:main}
	For every~$\varepsilon >0$ there exists~$\delta > 0$ such that, if~$p,q \in (\varepsilon,1-\varepsilon)$, then
	\begin{equation*}
	\mathbb{P}(|\mathscr{C}_{p,q}| = \infty) \leq \mathbb{P}(|\widehat{\mathscr{C}}_{p,q-\delta}| = \infty). 
	\end{equation*}
\end{proposition}

To allow a comparison between the percolation configurations of~$\mathbb{T}$ and~$\widehat{\mathbb{T}}$
we define the following functions:
\begin{align*}
\hat h: \widehat V\to \mathbb{N} \hspace{20pt}&\hat h(\hat o)=0,\\
&\hat h((\hat v_1,\ldots, \hat v_n))=n,\\
\Phi: \widehat{V} \to V\hspace{20pt} &\Phi((\hat{v}_1, \dots, \hat{v}_n))=(\varphi(\hat{v}_1)\cdot\varphi(\hat{v}_2)\cdot \dots \cdot \varphi(\hat{v}_n)),
\end{align*}
so that~$\hat h$ can be understood as height function on~$\widehat{\mathbb{T}}$ and
$\Phi$ is a surjective map between the vertex sets of the two graphs.
Note that~$\Phi$ maps the endpoints of an edge in~$\widehat{E}_{\s}$ and~$\widehat{E}_{\ell}$ into the endpoints of a short and a long edge in~$\mathbb{T}$ respectively.
Further, for any~$\hat i\in[d+d^k]$
\begin{align}
&h(\Phi(\hat i))= 1 \text{ or } k,\label{eq:heightfunc1}\\
&h(\Phi(\hat v \cdot \hat i))=h(\Phi(\hat v)) + h(\Phi(\hat i)). \label{eq:heightfunc2}
\end{align}
Consequently for~$\hat v= (\hat v_1, \dots, \hat v_n)$,
\begin{equation}\label{eq:phiheight}
h(\Phi(\hat v))=|\{i: \varphi(\hat v_i)\in [d]\}| + k \cdot |\{j: \varphi(\hat v_j)\in [d]^k\}|\geq \hat h(\hat v).
\end{equation}

Let~$\Lambda= (V_\Lambda, E_\Lambda)$ and~${\widehat\Lambda}= (\widehat{V}_\Lambda,\widehat{E}_\Lambda)$ be the subgraphs of~$\mathbb{T}$ and~$\widehat{\mathbb{T}}$ respectively induced by the edge sets
\begin{align*}
&E_\Lambda=\{\langle r, r\cdot i\rangle\in E: h(r)<2k\},\\
&\widehat{E}_\Lambda=\{\langle \hat{r}, \hat{r}\cdot\hat i\rangle\in\widehat{E}: h(\Phi(\hat{r}))<2k\}.
\end{align*}
Observe that~$V_\Lambda=\{v\in V: h(v)<3k\}$, furthermore~\eqref{eq:heightfunc1} and~\eqref{eq:heightfunc2} implies that if~$\hat{r}$ is the endpoint of an edge of~$\widehat{E}_\Lambda$, then~$h(\Phi(\hat{r}))<3k$.
Defining a leaf as a vertex with out-degree zero it is easy to see that
the set of leaves in~$\Lambda$ and~$\widehat\Lambda$ are
\begin{align}
&V_\Lambda^\text{leaf}=\{v \in V_\Lambda: 2k \leq h(v)<3k\},\label{eq:lambdaleaves}\\
&\widehat{V}_\Lambda^\text{leaf}=\{\hat v \in \widehat{V}_\Lambda: 2k \leq h(\Phi(\hat{v}))<3k\}.\label{eq:hatlambdaleaves}
\end{align}
In addition,~\eqref{eq:phiheight} implies that any path in~$\widehat\Lambda$ between the root and a leaf contains at most two edges of~$\widehat{E}_{\ell}$. Observe that
\begin{align}
&\Phi(\widehat{V}_\Lambda \setminus \widehat{V}_\Lambda^\text{leaf})=V_\Lambda\setminus V_\Lambda^\text{leaf},\label{eq:phimapslambda} \\
&\Phi(\widehat{V}_\Lambda^\text{leaf})=V_\Lambda^\text{leaf}.\label{eq:phimapsleaves}
\end{align}

Define~$\mathcal{C}_{p,q}$ and~$\widehat{\mathcal{C}}_{p,q}$ to be the cluster of the root in~$\Lambda$ and~$\widehat\Lambda$ respectively,
and define the random variables~$Z(\mathcal{C}_{p, q}):=|\mathcal{C}_{p,q}\cap V_\Lambda^\text{leaf}|$ and ~$\widehat Z(\widehat{\mathcal{C}}_{p, q}):=|\widehat{\mathcal{C}}_{p,q}\cap\widehat{V}_\Lambda^\text{leaf}|$,
that is the number of vertices in~$V_\Lambda^\text{leaf}$ and~$\widehat{V}_\Lambda^\text{leaf}$ that can be reached by an open path from the root. Clearly
\begin{equation}\label{eq:lambdanoofleaves}
Z(\mathcal{C}_{p, q})\leq |V_\Lambda^\text{leaf}|.
\end{equation}
Let~$\mathbf Z_{p, q}=(\mathbf Z_{n, p, q})_{n\geq 0}$ and~$\widehat{\mathbf Z}_{p, q}=(\widehat{\mathbf Z}_{n, p, q})_{n\geq 0}$ be (one-type) branching processes with offspring distribution~$Z(\mathcal{C}_{p,q})$ and~$\widehat Z(\widehat{\mathcal{C}}_{p,q})$ respectively.

We now state two lemmas that we will need for the proof of Proposition~\ref{prop:main}. The following result is elementary, we omit the proof.

\begin{lemma}\label{lemma:PZcoupling} We have
\begin{align}
 &\mathbb{P}(|\mathscr{C}_{p, q}| = \infty) \leq \mathbb{P}(\mathbf Z_{p, q} \text{ survives}) \text{ and }\label{eq:CZcoupling}\\
 &\mathbb{P}(|\widehat{\mathscr{C}}_{p, q}| = \infty) = \mathbb{P}(\widehat{\mathbf Z}_{p, q} \text{ survives})\label{eq:hatCZcoupling}.
 \end{align}

\end{lemma}

\begin{lemma}\label{lemma:ZhatZcoupling}
	For every~$\varepsilon > 0$ there exists~$\delta > 0$ such that for any~$p, q \in (\varepsilon, 1-\varepsilon)$, we have that~$Z(\mathcal{C}_{p,q})$ is stochastically dominated by~$ \widehat Z(\widehat{\mathcal{C}}_{p,q-\delta})$.
\end{lemma}
We postpone the proof of this lemma, and for now we show how the above results imply Proposition~\ref{prop:main}.
\begin{proof}[Proof of {Proposition~\ref{prop:main}}] Fix~$\varepsilon > 0$, and take~$\delta > 0$ corresponding to~$\varepsilon$ in Lemma~\ref{lemma:ZhatZcoupling}. For any~$p,q \in (\varepsilon,1-\varepsilon)$ we have
	\[
	\mathbb{P}(|\mathscr{C}_{p,q}| = \infty) \stackrel{\eqref{eq:CZcoupling}}{\leq}\mathbb{P}(\mathbf Z_{p, q} \text{ survives})  \le \mathbb{P}(\widehat{\mathbf Z}_{p, q-\delta} \text{ survives})\stackrel{\eqref{eq:hatCZcoupling}}{=} \mathbb{P}(|\widehat{\mathscr{C}}_{p,q-\delta}| = \infty).
	\]
\end{proof}
\subsection{Coupling and proof of Lemma~\ref{lemma:ZhatZcoupling}}
We will use the following coupling result from~\cite{LRV17}.
\begin{lemma}\label{lemma:coupling}
	Let~$\mathbb{P}_{\theta}$ denote probability measures on a finite set~$S$, parametrized by~$\theta\in(0, 1)^N$, and such that~$\theta\to \mathbb{P}_{\theta}(x)$ is continuous for every~$x\in S$.
	Assume that for some~$\theta_1$ and~$\bar{x} \in S$ we have~$\mathbb{P}_{\theta_1}(\bar{x})>0$.
	Then, for any~$\theta_2$ close enough to~$\theta_1$, such that
	\begin{equation}\label{eq:couplinglemma}
	\sum_{x\in S} |\mathbb{P}_{\theta_1}(x)-\mathbb{P}_{\theta_2}(x)|<\mathbb{P}_{\theta_1}(\bar{x}),
	\end{equation}
	there exists a coupling of two random elements~$X$ and~$Y$ of~$S$ such that 
	~$X\sim\mathbb{P}_{\theta_1}$, $Y\sim\mathbb{P}_{\theta_2}$ and 
	\[
	\mathbb{P}\left(  \{X=Y\}\cup\{X=\bar{x}\}\cup\{Y=\bar{x}\}\right)=1.
	\]
\end{lemma}

\begin{proof}[Proof of Lemma~\ref{lemma:ZhatZcoupling}]
	Let~$\widehat\Omega=\{0,1\}^{\widehat{E}_\Lambda}$ be the finite set of all possible configurations on~$\widehat\Lambda$.
	For a configuration~$\omega \in \widehat\Omega$ denote by~$\widehat{\mathcal{C}}(\omega)$ the cluster of the root on~$\widehat\Lambda$.
	We will give an algorithm to produce from~$\omega$ a connected subgraph~$\mathcal{C}(\omega)$ of~$\mathbb{T}_{d,k}$ containing the root.
	The construction will satisfy the following:
	if~$\omega$ is obtained from the product Bernoulli measure in which~$\omega(\hat e) = 1$ with probability~$p$
	if~$\hat e\in\widehat E_{\s}$ and with probability~$q$ if~$\hat e\in\widehat E_{\ell}$,
	then~$\mathcal{C}(\omega)$ and~$\widehat{\mathcal{C}}(\omega)$ are distributed as~$\mathcal{C}_{p,q}$ and~$\widehat{\mathcal{C}}_{p,q}$ respectively.
	
	The algorithm simultaneously explores the clusters~$\widehat{\mathcal{C}}(\omega)$ and~$\mathcal{C}(\omega)$ using alternating rounds of short and long edges.
	\begin{enumerate}
		\item Explore~$\widehat{\mathcal{C}}(\omega)$ querying only edges of~$\widehat{E}_{\s}$.
		Namely, starting from the root, at each step reveal an edge~$\langle\hat v, \hat r\rangle\in \widehat{E}_{\s}$ where~$\hat r$ is not in the cluster yet.
		For each such open edge add~$\Phi(\hat r)$ and~$\langle\Phi(\hat v), \Phi(\hat r)\rangle$ to~$\mathcal{C}(\omega)$.
		Continue until no further vertex can be reached using only edges of~$\widehat{E}_{\s}$.
		Note that after this step~$\mathcal{C}(\omega)$ only contains short edges.
		\item Continue the exploration of~$\widehat{\mathcal{C}}(\omega)$ querying edges of~$\widehat{E}_{\ell}$.
		That is, for each~$\hat v$ in~$\widehat{\mathcal{C}}(\omega)$ so far, reveal edges~$\langle\hat v, \hat r\rangle\in \widehat{E}_{\ell}$.
		If an edge is open, add~$\langle\Phi(\hat v), \Phi(\hat r)\rangle$ to~$\mathcal{C}(\omega)$.
		At this point it is possible that~$\Phi(\hat r)$ is already in~$\mathcal{C}(\omega)$;
		in this case we say that vertex~$\hat r$ causes conflict, and we do not explore its subtree in the upcoming steps.
		Otherwise, add~$\Phi(\hat r)$ to~$\mathcal{C}(\omega)$.
		Continue until no further vertex can be reached using only edges of~$\widehat{E}_{\ell}$.
		In this step we only add long edges to~$\mathcal{C}(\omega)$.
		\item Repeat the exploration process of Step 1 starting from the vertices that were added in Step 2. (Keep in mind that these are vertices that did not cause any conflict.)
		Again, if a new vertex causes conflict, do not continue the exploration process on its subtree.
		\item For the vertices that were added in Step 3 (hence did not cause conflict), repeat Step 2.
	\end{enumerate}
	Note that this algorithm does not necessarily explore the whole cluster~$\widehat{\mathcal{C}}(\omega)$, as it stops at vertices that cause conflict.
	By~\eqref{eq:phimapslambda} and~\eqref{eq:phimapsleaves} each leaf in~$\mathcal{C}(\omega)$ corresponds to an explored vertex of~$\widehat{V}_\Lambda^\text{leaf}$.
	Therefore
	\begin{equation}\label{eq:algleaves}
	Z(\mathcal{C}(\omega))\leq \widehat Z(\widehat{\mathcal{C}}(\omega)).
	\end{equation}
	
	Now for a fixed~$p, q$ and~$\delta$ close enough to zero, we will define a coupling measure~$\mu$ on~$\widehat\Omega^2$ satisfying
	\begin{align*}
	(\omega, \omega')\sim\mu\hspace{20pt} \Longrightarrow \hspace{20pt} &\mathcal{C}(\omega) \stackrel{\text{(d)}}{=}\mathcal{C}_{p, q}, \hspace{5pt} \widehat{\mathcal{C}}(\omega')\stackrel{\text{(d)}}{=}\widehat{\mathcal{C}}_{p, q-\delta}\text{ and } Z(\mathcal{C}(\omega))\leq \widehat Z(\widehat{\mathcal{C}}(\omega')).
	\end{align*}
	The construction will involve Lemma~\ref{lemma:coupling}.
	
	Define a configuration~$\bar{\omega}\in\widehat\Omega$ as follows:
	\begin{itemize}
		\item every edge starting from the root is open;
		\item for every~$\hat v$ satisfying~$\langle \hat o, \hat v \rangle\in\widehat E_{\ell}$, every edge on the subtree of~$\hat v$ is open;
		\item for every~$\hat v$ satisfying~$\langle \hat o, \hat v \rangle\in\widehat E_{\s}$, on the subtree of~$\hat v$ only edges of the form
		$\{\langle \hat r, \hat s\rangle\in\widehat E_{\s}: \hat h(\hat s)\leq k\}$ are open.
	\end{itemize}
	Now we construct~$\mathcal{C}(\bar\omega)$ following the algorithm above. The first step of the exploration reveals edges~$\{\langle \hat r, \hat s\rangle\in\widehat E_{\s}: \hat h(\hat s)\leq k\}$ of~$\widehat{\mathcal{C}}(\bar\omega)$.
	The corresponding open edges in~$\mathcal{C}(\bar\omega)$ are~$\{\langle r,  s\rangle\in E_{\s}: h(s)\leq k\}$, hence every vertex~$v$ in~$V_\Lambda$ satisfying~$h(v)<k+1$ will be added to the cluster.
	In Step 2 we explore the long edges starting from the root. 
	The endpoints of these edges are mapped into vertices of~$V_\Lambda$ at height~$k$, which are already in~$\mathcal{C}(\bar\omega)$.
	Thus these vertices cause conflicts, their subtrees will not be explored.
	By the definition of~$\bar\omega$ there are no more long edges starting from the cluster explored so far, so the exploration process stops.
	Since the vertices that can be reached by an open path in~$\mathcal{C}(\bar\omega)$ satisfy~$h(v)<k+1$,
	by~\eqref{eq:lambdaleaves} we have~\[Z(\mathcal{C}(\bar\omega))= 0.\]
	
	Now we will show that for any~$v\in V_\Lambda^\text{leaf}$ there exists a~$\hat v\in \widehat{\mathcal{C}}_{p,q}(\bar{\omega})\cap\widehat{V}_\Lambda^\text{leaf}$
	such that~$\Phi(\hat v)=v$, therefore~\[\widehat Z(\widehat{\mathcal{C}}(\bar{\omega}))\geq |V_\Lambda^\text{leaf}|.\]
	Let~$v=(v_1, \dots, v_m)\in V_\Lambda^\text{leaf}$, then by~\eqref{eq:lambdaleaves} we have~$2k\leq m <3k$.
	Define~$\hat v=(\hat v_1, \dots, \hat v_{m-2k})$ as follows:
	\begin{align*}
	\hat v_1=&\varphi^{-1}((v_1, \dots, v_k)),\\
	\hat v_2=&\varphi^{-1}(v_{k+1}),\\
	&\vdots\\
	\hat v_{m-2k+1}=&\varphi^{-1}(v_{m-k}),\\
	\hat v_{m-2k}=&\varphi^{-1}((v_{m-k+1}, \dots, v_m)).\\
	\end{align*}
	Then~$\Phi(\hat v)=v$ and by~\eqref{eq:hatlambdaleaves}~$\hat v \in \widehat{V}_\Lambda^\text{leaf}$.
	By the definition of~$\bar\omega$ there is an open path to~$\hat v$ in~$\bar\omega$, hence~$\hat v\in \widehat{\mathcal{C}}_{p,q}(\bar{\omega})\cap\widehat{V}_\Lambda^\text{leaf}$ indeed.
	
	By Lemma~\ref{lemma:coupling}, if~$\delta$ is small enough, there exists a coupling of configurations~$X, Y\in \widehat\Omega$ satisfying
	\begin{itemize}
		\item the values of~$X$ on all edges are independent;
		\item~$X$ assigns each edge of~$\widehat E_{\s}$ and~$\widehat E_{\ell}$ to be open with probability~$p$ and~$q$ respectively; 
		\item the values of~$Y$ on all edges are independent;
		\item~$Y$ assigns each edge of~$\widehat E_{\s}$ and~$\widehat E_{\ell}$ to be open with probability~$p$ and~$q-\delta$ respectively;
		\item~$(X, Y)$ satisfies
		\begin{equation}\label{eq:coupling}
		\mathbb{P}\left(  \{X=Y\}\cup\{X=\bar{\omega}\}\cup\{Y=\bar{\omega}\}\right)=1.
		\end{equation}
	\end{itemize}
	
	Now let~$\omega=X$ and~$\omega'=Y$.
	Let us examine~$Z(\mathcal{C}(\omega))$ and~$\widehat Z(\widehat{\mathcal{C}}(\omega'))$ in all possible cases listed in~\eqref{eq:coupling}:
	\begin{itemize}
		\item if~$X=Y$, then by~\eqref{eq:algleaves} we have~$Z(\mathcal{C}(\omega))\leq \widehat Z(\widehat{\mathcal{C}}(\omega'))$;
		\item if~$X=\bar\omega$, then~$Z(\mathcal{C}(\omega))= 0$;
		\item if~$Y=\bar\omega$, then~$\widehat Z(\widehat{\mathcal{C}}(\omega'))\geq\ |V_\Lambda^\text{leaf}| \stackrel{\eqref{eq:lambdanoofleaves}}{\geq} Z(\mathcal{C}(\omega))$ for all~$\omega$.
	\end{itemize}
	Hence in all cases~$\widehat Z(\widehat{\mathcal{C}}(\omega'))\geq Z(\mathcal{C}(\omega))$.
	
\end{proof}
\vspace{0.5cm}

\textbf{Acknowledgements.} The research of B.N.B.L.\ was supported in part by CNPq grant 305811/2018-5 and FAPERJ (Pronex E-26/010.001269/2016). He is also thankful for the hospitality of Bernoulli Institute, University of Groningen. The research of R. Sz.\ was partially supported by ERC Starting Grant 680275 ``MALIG''. The authors would like to thank Tobias M\"uller for raising the question that lead to Theorem~\ref{THM:CRITCURVE}, and Aernout van Enter for suggesting useful references.

\end{document}